\documentclass[12pt,oneside]{amsart}
\usepackage[utf8]{inputenc}
\usepackage[english]{babel}

\usepackage[margin=1in]{geometry}

\usepackage{multicol}

\usepackage{amsmath}
\usepackage{amsthm}
\usepackage{amssymb}
\usepackage{graphicx}
\usepackage{mathrsfs}
\usepackage[colorinlistoftodos]{todonotes}
\usepackage[colorlinks=true, allcolors=black]{hyperref}
\usepackage{comment}
\usetikzlibrary{graphs}
\usepackage{float}
\usepackage{xcolor}
\usepackage{soul}
\usepackage{tikz-cd}

\usepackage{mdwlist}

\usepackage{phaistos}

\newcommand{\showcomments}{yes}
\newcommand{\showrefcomments}{yes}
\renewcommand{\showrefcomments}{no}

\newsavebox{\commentbox}
{\ifthenelse{\equal{\showcomments}{yes}}%
{\footnotemark
        \begin{lrbox}{\commentbox}
        \begin{minipage}[t]{1.5in}\raggedright\sffamily\tiny
        \footnotemark[\arabic{footnote}]}
{\begin{lrbox}{\commentbox}}}%
{\ifthenelse{\equal{\showcomments}{yes}}%
{\end{minipage}\end{lrbox}\marginpar{\usebox{\commentbox}}}
{\end{lrbox}}}

{\ifthenelse{\equal{\showrefcomments}{yes}}%
{\footnotemark
        \begin{lrbox}{\commentbox}
        \begin{minipage}[t]{1.25in}\raggedright\sffamily\tiny
        \footnotemark[\arabic{footnote}]}
{\begin{lrbox}{\commentbox}}}%
{\ifthenelse{\equal{\showrefcomments}{yes}}%
{\end{minipage}\end{lrbox}\marginpar{\usebox{\commentbox}}}
{\end{lrbox}}}

\usepackage{tikz}

\newtheorem{thm}{Theorem}[section]
\newtheorem{lem}[thm]{Lemma}

\newtheorem{prop}[thm]{Proposition}

\newtheorem{theorem}[thm]{Theorem}
\newtheorem{corollary}[thm]{Corollary}
\newtheorem{lemma}[thm]{Lemma}
\newtheorem{proposition}[thm]{Proposition}

 \newtheorem{question}[thm]{Question}

\theoremstyle{definition}
\newtheorem{defn}[thm]{Definition}
\newtheorem{definition}[thm]{Definition}

\theoremstyle{remark}
\newtheorem{rem}[thm]{Remark}

\newtheorem{warning}[thm]{Warning}

\newtheorem{notation}[thm]{Notation}

\newtheorem{remark}[thm]{Remark}

\newtheorem{construction}[thm]{Construction}

\newcommand{\dist}{\textup{\textsf{d}}}
\newcommand{\scname}[1]{\text{\sf #1}}
\newcommand{\area}{\scname{Area}}
\newcommand{\p}{\textup{\textsf{p}}}

\DeclareMathOperator{\diam}{\text{diam}}

\usepackage{xcolor}

\definecolor{dani}{rgb}{1,0.5,.3}
\definecolor{maca}{rgb}{0,0.6,1}
\definecolor{definition}{rgb}{0,0,0}
\definecolor{statement}{rgb}{0,0,0}
\definecolor{proof}{rgb}{0, 0, 0}
\definecolor{comment}{rgb}{0, 0, 0}

\newcommand{\curvature}{{\ensuremath{\kappa}}}
\newcommand{\field}[1]{\mathbb{#1}}
\newcommand{\integers}{\ensuremath{\field{Z}}}
\newcommand{\rationals}{\ensuremath{\field{Q}}}
\newcommand{\naturals}{\ensuremath{\field{N}}}
\newcommand{\reals}{\ensuremath{\field{R}}}

\newcommand{\euler}{\chi}

\DeclareMathOperator{\link}{link}
\DeclareMathOperator{\interior}{Int}
\DeclareMathOperator{\deficiency}{\textsf{def}}

\newcommand{\boundary}{\partial}

\newcommand{\neb}[1]{\mathcal N_{#1}}

\title{Linear isoperimetric functions for surfaces in hyperbolic groups}

\begin{document}

\author{Macarena Arenas}
\address{Dept. of Math. \& Stats., McGill Univ., Montreal, QC, Canada H3A 0B9 \\ \textit{Current address:} DPMMS, Centre for Mathematical Sciences, Wilberforce Road, Cambridge, CB3 0WB, UK}
\email{mcr59@dpmms.cam.ac.uk}

\author{Daniel T. Wise}
\address{Dept. of Math. \& Stats., McGill Univ., Montreal, QC, Canada H3A 0B9}
\email{wise@math.mcgill.ca}

%    General info

\maketitle

\begin{abstract}
We show that word-hyperbolic groups satisfy linear isoperimetric functions for all homotopy types of surface diagrams. This generalises the linear isoperimetric functions for disc and annular diagrams. 
\end{abstract}

%\tableofcontents

\section{Introduction}

One of the main characterisations of word-hyperbolic groups is that they are the groups that satisfy a
 linear isoperimetric inequality. That is, for a compact 2-complex $X$,
the hyperbolicity of $\pi_1X$ is equivalent to the existence of a linear isoperimetric function
for disc diagrams $D\rightarrow X$. This means that there is a constant $K$ 
such that if there exists a disc diagram $D\rightarrow X$, then there exists a disc diagram $D'\rightarrow X$ with
$\boundary_\p D'=\boundary_\p D$, and with $\area(D')\leq K|\boundary_\p D'|$.
It is likewise known that hyperbolic groups have a linear annular isoperimetric function.
The goal of this paper is to generalise the linear isoperimetric function to arbitrary surface diagrams. The ``$(g,n)$-isoperimetric function" for $X$ is the maximal area needed to fill in ``surface diagrams" with genus $g$ and $n$ boundary circles. 
 See Definition~\ref{def:isopfction}. 

\begin{thm}\label{thm:main intro}
Let $X$ be a compact cell complex with $\pi_1X$ torsion-free hyperbolic. 
For each $g$ and $n$, the $(g,n)$-isoperimetric function for $X$ is linear.
\end{thm}

This is a slightly simplified version of the statement. The precise formulation of Theorem~\ref{thm:main intro} is given in Theorem~\ref{thm:main}, and only requires that boundary circles of surface diagrams do not represent non-trivial torsion elements.

A motivating case is Gromov's beautiful observation (\cite{Gromov87})  
that if $n$ closed local geodesics $C_i\rightarrow M$ in a compact negatively curved manifold $M$ form the boundary circles of a  surface $S \rightarrow M$,
then one can rechoose $S$ so that $\area(S)=2\pi(2g+n-2)$. 
 Gromov's argument generalises to a compact negatively curved space $X$ with $\pi$ replaced by an upper bound on the area of an ideal triangle in $\widetilde X$.
One would hope that Gromov's method of proof could be adapted seamlessly to arbitrary hyperbolic groups, but this is apparently not the case:  there are fundamental technical barriers to such a generalisation. Perhaps most salient among these is the fact that the sides of ideal triangles in a $\delta$-hyperbolic complex do not necessarily asymptotically converge - and from a combinatorial viewpoint, ideal triangles do not bound diagrams with finite area.

To prove Theorem~\ref{thm:main intro} in its full generality, we analyse a surface decomposition similar to that in Proposition~\ref{prop:gromovs trick},
but are then diverted into the issue of extracting an appropriate compact surface from this procedure, since the decomposition cannot produce a compact surface with boundary directly.
We ultimately explain how to use the decomposition to estimate the combinatorial area and arrive at our goal of a linear isoperimetric function for surface diagrams.

At this point we assume, as in Proposition~\ref{prop:gromovs trick} below, that the boundary circles of $S$ are essential in $X$, and moreover, represent conjugacy classes of elements having infinite order in $\pi_1X$.

\subsection{Gromov's trick}
In this section we recall Gromov's use of an ideal triangle decomposition to bound the area of a surface  \cite[p.235]{Gromov87} in a negatively curved Riemannian manifold. 
Thurston originally used such ideal triangle decompositions
of hyperbolic surfaces with boundary to build examples, and also glued ideal polyhedra together to form hyperbolic 3-manifolds
\cite{Thurston82}.  We state Gromov's result for genus~$0$ surfaces, as presented in \cite{Gromov87}, but note that the proof works whenever $\chi\leq 0$. 

\begin{prop}\label{prop:gromovs trick}
Let $M$ be a compact negatively curved manifold with boundary.
Let $S$ be a compact surface with $\boundary S=\sqcup_{i=1}^n C_i$ and $n\geq3$. 
There exists a constant $K$ with the following property:

Let $S\rightarrow M$ be a map such that each $C_i\rightarrow M$ is essential.
Then $S\rightarrow M$ can be homotoped to another mapped surface  $S'\rightarrow M$ such that:
\begin{enumerate}
\item  $\interior(S')$ is built from the union of $2(n-2)$ ideal triangles in $\widetilde M$,
\item $\area(S')\leq K(|\partial S| -\chi(S))$,
\item  $\boundary S' = \sqcup _{i=1}^n C_i'$
and each $C_i'\rightarrow M$ is a local geodesic.
\end{enumerate}
\end{prop}
  
\begin{proof}[Sketch]

Build $S'$ in two steps: first homotope each $C_i$ in $\boundary S$ to the unique closed geodesic $C_i'$ in its homotopy class to obtain a ``collar" composed of cylinders which can be chosen so that the area of the collar is bounded above by a linear function of $|\partial S|$ using the isoperimetric inequality for annuli. 

We thus obtain a surface with geodesic boundary $\partial S'$, and this surface can be homotoped, keeping the boundary fixed, to a surface whose area is $\leq K |\chi(S')|$ by decomposing it into ideal triangles whose sides, in the universal cover, lift to lines that are asymptotic to geodesics covering boundary components. These geodesics project to geodesics in $M$, so $\partial S'$ can be filled in with $m$ ideal triangles, where the number $m$ depends on $\chi(S')=\chi(S)$.
\end{proof}

\begin{rem}
The explanation in Proposition~\ref{prop:gromovs trick}
functions perfectly well for a space that is negatively curved
in the sense that it is locally CAT($\kappa$) for some $\kappa<0$.
However, there are two important points to consider.
Firstly, there are foundational issues relating the area of the constructed surface 
to the combinatorial area of a diagram, which is what we shall pursue.
(See \cite{Bridson2002,BurilloTaback2002} for work relating
classical isoperimetric area to combinatorial area.)
Secondly, we are interested in providing a general linear isoperimetric function 
in the case where $M$ is nonpositively curved but $\pi_1M$ is word-hyperbolic,
and more generally, where $M$ does not even have a locally CAT(0) metric.
A substantial technical obstacle is that outside the negatively curved case,
there could be flat strips, and hence ideal triangles might not behave in a fashion allowing us to produce a compact surface.
\end{rem}

\subsection{Related results}

There is a stream of research pursuing a homological alternative to Theorem~\ref{thm:main intro}. 

Hyperbolic groups satisfy a \emph{weak}  (sometimes called \emph{homological}) linear isoperimetric inequality, in the sense that a $k$-cycle that bounds a $(k+1)$-chain bounds such a chain whose ``area" is linear on the ``length" of the $k$-cycle.  This notion was first discussed by Gersten in \cite{Gersten1996}. The existence of a weak linear isoperimetric inequality for hyperbolic groups was proven for $k=1$ in \cite{Gersten1998} and extended to all $k\geq1$ by Mineyev  \cite{Mineyev2000} (for homology with $\rationals$ and $\reals$ coefficients) and by Lang  \cite{Lang2000} (for $\integers$ coefficients). 
In the $k=1$ case, which is the case most closely related to this paper, the above result says that a collection of cellular loops that bounds an orientable surface, bounds such a surface (possibly with very large genus) whose area is linear on the length of the loops. 

Our result is stronger, since we show that given a genus~$g$ diagram with prescribed boundary circles, there exists a genus~$g$ diagram   having exactly those circles as boundary components --so, in particular, of the same homotopy type
\footnote{If the input surface diagram is compressible in $X$,  then the output will be a disconnected surface diagram. However, since $X$ is path-connected, it is possible to upgrade this surface diagram to a 2-complex that is homotopy equivalent to the original surface diagram. The caveat is that the resulting complex will \emph{not} be a surface diagram as per our definition.}
as the original diagram -- and whose area is linear on the length of the circles. Moreover, our methods have the added advantage of encompassing nonorientable surface diagrams as well.

Although the homological and ordinary isoperimetric functions are both linear for hyperbolic groups, they can be inequivalent for arbitrary finitely presented groups \cite{ABDY13}. A characterisation of relative hyperbolicity using a weak isoperimetric function was given in \cite{MPE16}.

In a very different direction, there are results controlling the areas of cylinders associated to simultaneous conjugations in hyperbolic groups \cite{BridsonHowie05}, \cite{BH13}. 

\subsection{Further directions}
 
The main Theorem requires that the boundary circles of the surface diagram map to conjugacy classes represented by infinite order elements of $\pi_1 X$. Our proof requires this hypothesis in order to create the ideal triangle decomposition. Nevertheless, we hope that there might be some variant construction supporting a generalisation:

\begin{question}
Does the statement of Theorem~\ref{thm:main intro} hold without the requirement that each $C_i$ maps to a conjugacy class of an infinite order element of $\pi_1 X$?
\end{question}

Our Theorem indicates that the classical, homological, and generalised isoperimetric functions are all equivalent for hyperbolic groups. This is not the case outside of the hyperbolic setting: already in the class of groups having quadratic classical isoperimetric functions, there are examples having unsolvable conjugacy problem. Such examples, in particular, cannot even satisfy a recursive annular isoperimetric function. This is a result of Olshanskii and Sapir~\cite{OS20} that negatively answers a question posed by Rips. 

In view of these results, it seems reasonable to ask:

\begin{question}
Let $G$ be a group in some favourite geometric group theory class, what are the $(g,n)$-isoperimetric functions for $G$?
\end{question}

Our proof provides a linear isoperimetric function $f_{gn}$ depending on the genus $g$ and the number $n$ of boundary circles.
It is natural to ask whether Theorem~\ref{thm:main intro} can be uniformised: 

\begin{question}\label{quest2}
Is there a ``global" linear isoperimetric function $F:\naturals \rightarrow [0, \infty)$ with $f_{gn}(m)\leq F(m)$ for all $m\in \naturals$?
\end{question}

In the special case of complexes that satisfy the strict weight test (Section~\ref{sec:special}), we obtain inequalities that only depend on $g$. This suggest that in more restricted combinatorial settings Question~\ref{quest2} may be more tractable. In contrast, Proposition~\ref{prop:gromovs trick} indicates that Question~\ref{quest2} will not have an affirmative answer in general.

\subsection{Structure of the paper}

In Section~\ref{sec:class} we define surface diagrams and generalised isoperimetric functions, state the classical isoperimetric inequalities for disc and annular diagrams, and review some well-known lemmas controlling the behaviour of geodesics and quasigeodesics in hyperbolic spaces.

In Section~\ref{sec:linso} we state the main Theorem and outline the structure of the proof, we then proceed to prove a generalisation of the slim triangle property, and describe a number of constructions which will change and simplify the surface diagram in various ways.

In Section~\ref{subsec:horizontalpt} we define ``jumps", which allow us to describe a graph of spaces decomposition associated to the surface diagram. We then define ``horizontal paths" and ``bands", and use the graph of spaces structure to understand the combinatorics of these objects. This  provides us with a way to decompose a surface diagrams into annular and disc diagrams, and hence obtain a small-area surface diagram by controlling the area of the pieces in this decomposition. 

In Section~\ref{sec:special} we give a more elementary proof of the Theorem for $2$-complexes satisfying a strong local negative curvature condition. We do not know if there is a way to generalise this to handle arbitrary hyperbolic groups.

Finally, in Section~\ref{sec:nono} we deal with nonorientable surface diagrams. Although it would be possible to unify the orientable and nonorientable cases, we felt that this would not provide any additional intuition or clarity regarding the methods involved, so we opted instead to present it separately.

\section{Classical statements}\label{sec:class}

We recall some definitions and  classical results that will be needed throughout the paper:

\subsection{Diagrams}

\begin{definition}\label{def:diagrams}
A \emph{surface diagram} $S$ is a compact combinatorial $2$-complex with an embedding $S \subset \bar S$ into a surface with punctures  such that $\bar S$ deformation retracts to $S$. If $\bar S$ has genus $g$, then $S$ is a \emph{genus g diagram}. Let $\partial \bar S = \cup^n_{i=1} \bar C_i$. We shall always assume that $n\geq 1$.  

The surface diagram has $n$ \emph{boundary paths} $P_1, \dots, P_n$  corresponding to attaching maps of 2-cells that could be added to $S$ to form a closed surface. More precisely, each $P_i$ is homotopic to a $\bar C_i \subset \bar S$.  The \emph{boundary} of $S$ is the union $\partial S=\cup^n_{i=1} C_i$ where
each $C_i$ is homeomorphic to a circle and each $P_i$ maps to $C_i$.
%\begin{com} what about the singular case??
%\end{com}
A genus~$g$ diagram is  \emph{singular} if it is not homeomorphic to a surface (e.g.\ $S$ might have cut-vertices).
%We emphasize that a genus~$g$ diagram is not necessarily connected, but 
The most frequently considered cases have 
 $g=0$, and $S$ connected and orientable: the case $n=1$  is a \emph{disc diagram} and the case $n=2$ is an \emph{annular diagram}. 
  %A \emph{Möbius diagram} is the special case where the $2$-complex deformation retracts to a Möbius strip. 

A \emph{genus~g diagram in a complex X} is a combinatorial map $S \rightarrow X$ where $S$ is a genus~$g$ diagram.

Let $\area(S)$ denote the number of 2-cells in $S$. Let $|P|$ denote the combinatorial length of a path. 
\end{definition}

When $D$ is a disc diagram, we use the notation $\partial_\p D$ for the \emph{boundary path} of $D$, which is the path travelling around $D$ that corresponds to the attaching map of the 2-cell $\mathbb{S}^2 -D$, where $\mathbb{S}^2$ is obtained by removing the puncture. 

\begin{lemma}[Van Kampen]\label{lem:VK}
Let $X$ be a combinatorial $2$-complex. Let $P \rightarrow X^1$ be a closed combinatorial path. Then $P$ is nullhomotopic if and only if there exists a disc diagram $D$ with $\partial_\p D\cong P$ and a map $D \rightarrow X$ so that there is a commutative diagram:

\[\begin{tikzcd}
\partial_\p D \arrow[r] \arrow[d] & D \arrow[d] \\
P \arrow[r]                      & X          
\end{tikzcd}\]
\end{lemma}

\begin{remark}[Orientability]
 For Sections~\ref{sec:linso} through~\ref{sec:special}, all surface diagrams are assumed to be orientable. Non-orientable surface diagrams come into play only in Section~\ref{sec:nono}. 
\end{remark}

\subsection{Isoperimetry}

We now introduce the main object of interest:

\begin{definition}[Generalised isoperimetric functions]\label{def:isopfction}
Let $P_1 \rightarrow X, \ldots, P_n \rightarrow X$ be closed combinatorial paths. We define their ``genus $g$ area" by: \[\area_g(\sqcup_i P_i) \ = \ inf  \big( \area(S): S \rightarrow X \text{ is a genus $g$ diagram and }\partial S = \sqcup_i P_i \big)\]
The $(g,n)$-\emph{isoperimetric function} for $X$ is the function $f_{gn}:\naturals \rightarrow [0,\infty]$ defined as follows, where the supremum is taken over $\sqcup_i P_i$ having  $\area_g(\sqcup_i P_i) < \infty$:  
\[f_{gn}(m) \ = \ sup \big(\area_g(\sqcup_i P_i) :  \sum_i |P_i|=m \big)\]

\end{definition}

Let  $X$ be a compact 2-complex whose universal cover $\widetilde X$ has 1-skeleton
that is $\delta$-hyperbolic for some $\delta \geq 0$.
Proofs of the following results can be found in \cite[p.417 and p.454]{BridsonHaefliger}:

\begin{theorem}[Disc isoperimetry]\label{thm:discs}
There is a constant $N=N(\delta)$ such that for every null-homotopic closed combinatorial path $\sigma \rightarrow X$, there exists a disc diagram  $D \rightarrow X$ with $\partial_\p D=\sigma$ and $\area(D)\leq N|\sigma|$.
\end{theorem}

\begin{proposition}[Annular isoperimetry]\label{prop:BHconj}
There is a constant $M=M(\delta)$ such that 
if two  essential closed combinatorial paths $\sigma$ and $\sigma'$ are homotopic in $X$, there exists an annular diagram $A \rightarrow X$ with $\partial A= \sigma \cup \sigma'$ and $\area(A)\leq M \cdot max\{|\sigma|,|\sigma'|\}$. 
\end{proposition}

\subsection{More on hyperbolic spaces}

Gromov-hyperbolicity is also characterised by exponential divergence:

\begin{theorem}\label{thm:cdiv}
Let $X$ be a $\delta$-hyperbolic geodesic metric space, then there exists an exponential function $e: \naturals \rightarrow \reals$ with the following property.

For all $R,r\in \naturals$, all $x \in X$, and all geodesics $c_1:[0,a_1] \rightarrow X, c_2:[0,a_2] \rightarrow X$ with $c_1(0)=c_2(0)=x$,
if $R+r\leq min\{a_1,a_2\}$ and $\dist(c_1(R), c_2(R))> e(0)$, then any path connecting $c_1(R+r)$ to $c_2(R+r)$ outside the ball $B(x, R+r)$ must have length at least $e^r$. 
\end{theorem}

\begin{definition}
An $(a,b)$-\emph{quasigeodesic} (where $a > 0$ and $b \geq 0$) is a function $\varphi:\reals \rightarrow X$ satisfying the following for all $s,t \in \reals$: $$\frac{1}{a}|s-t|-b \ \leq \ d(\varphi(t),\varphi(s)) \ \leq \ a|s-t|+b.$$
\end{definition}

\begin{theorem}\label{thm:quasigg}
Let $p$ and $q$ be points of a $\delta$-hyperbolic geodesic metric space $X$. For each $a, b >0$, there exists a constant $L=L(a,b)$ such that the following holds: If $\sigma$ and $\sigma'$ are $(a,b)$-quasigeodesics with the same endpoints, then $\sigma \in \neb{L}(\sigma')$.
\end{theorem}

We use $\neb{\epsilon}(K)$ to denote the $\epsilon-$neighbourhood of $K$. Metric discussions of a complex $X$ actually refer to its 1-skeleton.

A \emph{geodesic ray} $\eta:[0, \infty) \rightarrow X$ is an isometric embedding. %It is \emph{geodesic} if it is an isometric embedding. 
Two geodesic rays $\eta, \eta'$ are \emph{equivalent} if there is a constant $L$ such that $\dist (\eta(t),\eta'(t))\leq L$ for all $t$. The \emph{Gromov boundary} of $X$ is the set $\partial X=\{[\eta] \ :\  \eta \text{ is a geodesic ray}\}$.

We employ the following consequence of Theorem~\ref{thm:cdiv}:

\begin{corollary}\label{cor:closediverge} Let $X$ be $\delta$-hyperbolic. For each $T,O\geq 0$ there exists $R$ such that if  $\gamma, \gamma'$ are geodesic rays representing distinct points of $\partial X$ and $\dist(\gamma(0),\gamma'(0))\leq O$, then $\dist(\gamma(r), \gamma'(r))\geq T$ for all $r\geq R$.
\end{corollary}

\begin{proof}
Let $s$ be a geodesic from $\gamma(0)$ to $\gamma'(0)$. 
Consider the $(1,O)$-quasigeodesic $s\gamma'$ . By Theorem~\ref{thm:quasigg} there is a constant $L$ and a geodesic ray $\gamma''$ with $s\gamma' \subset N_L(\gamma'')$ and $\gamma''(0)=\gamma(0)$. Hence there is $L'$ such that $\dist(\gamma'(t), \gamma''(t))\leq L'$ for all $t\geq 0$. By Theorem~\ref{thm:cdiv} there is an exponential function $e^t$ and a constant $t_0 >0$ such that for $t >t_0$ we have $\dist(\gamma(t), \gamma''(t)) \geq e^{t-t_0}$. But $s\gamma' \subset N_L(\gamma'')$, so $s\gamma'$, and hence $\gamma'$, is at distance $e^{t-t_0}$ from $\gamma(t)$.
The conclusion follows since 
$\dist(\gamma(t), \gamma'(t))+L' \geq
\dist(\gamma(t),\gamma'(t))+\dist(\gamma'(t), \gamma''(t))
\geq \dist(\gamma(t), \gamma''(t)) \geq e^{t-t_0}$.
\end{proof}

We will make use of the following local-to-global criterion for quasigeodesics. A proof can be found in \cite{HagenWiseFreeByZgeneral} in a slightly different setting.

\begin{theorem}\label{thm:loco}
Let $X$ be $\delta$-hyperbolic. %Let $\gamma$ be a geodesic with the same endpoints as the
Consider a piecewise geodesic path
$\sigma_1\lambda_1\sigma_2\lambda_2 \cdots\lambda_k\sigma_{k+1}$.
For each $L>0$ there exists $\alpha,\beta>0$  such that
 $\sigma_1\lambda_1 \cdots\lambda_k\sigma_{k+1}$ is a $(\alpha,\beta)$-quasigeodesic
 provided that:
\begin{enumerate}
\item $\frac{1}{2}|\lambda_i|\geq 6(L+\delta)$ for each $i$.
\item $\diam\big(\lambda_i \cap \neb{3\delta}(\lambda_{i+1})\big)\leq L$ for each $i$.
\item $\diam\big(\lambda_i \cap \neb{3\delta}(\sigma_{i+1})\big)\leq L$ for each $i$.
\item $\diam\big(\sigma_i \cap \neb{3\delta}(\lambda_i)\big)\leq L$ for each $i$.
\end{enumerate}
\end{theorem}

\section{Linearity of generalised isoperimetric functions}\label{sec:linso}

\begin{theorem}\label{thm:main}
Let $X$ be a compact 2-complex such that the 1-skeleton of $\widetilde{X}$ is $\delta$-hyperbolic. For each $n\geq 1$ there is a constant $k_{gn}$ such that the following holds:
Let $S\rightarrow X$ be a genus~$g$ diagram in $X$
 with boundary circles $C_1,\dots,C_n$ and suppose each $C_i \rightarrow X$ is either null-homotopic or represents an infinite-order element of $\pi_1X$. 
 There exists %a constant $K_n>0$ and 
 a genus~$g$ diagram $S'\rightarrow X$ with $\partial S=\partial S'$ and  $\area(S')\leq k_{gn} |\boundary S'|$.
\end{theorem}

\begin{proof}
{\bf Organisation of the proof:}
We first handle the following situations:
\begin{enumerate}
\item \label{item:1}The case where $g=0$ and $n=1$ is Theorem~\ref{thm:discs}.
\item \label{item:2}The case where $g=0$ and $n=2$ is Proposition~\ref{prop:BHconj}.
\item \label{item:3}If $S$ is disconnected, proving the result for each component implies the result for $S$. If $S$ has cut vertices, let $V$ be the set of cut-vertices and let $S'=S - V$. This is a disconnected surface, and each of its components may be viewed as a compact surface with boundary by gluing back the relevant cut vertices to it. We can therefore prove the Theorem for each component $S_j$ and take $S'=\cup S_j$. 
\item \label{item:4}If $S$ has a null-homotopic circle $c$ that is not a boundary circle, then realising the homotopy there is a singular surface diagram $S'$ having the same boundary circles as $S$, and~\eqref{item:3} applies.  
\item \label{item:5} If $S$ has a null-homotopic boundary circle $C_n$, 
then by Theorem~\ref{thm:discs} there is a disc diagram $D\rightarrow X$ with $\boundary D=C_n$
and $\area(D)\leq k_{01}|C_n|$.
Letting $T=D\cup_{C_n} S$,
by induction there exists $T'$ with $\boundary T'=\boundary T$ and $\area(T')\leq k_{g(n-1)}|\partial T'|$.
Hence letting $S'=D\sqcup T'$ the result holds with $k_{gn}\geq \max\{k_{g1},k_{g(n-1)}\}$.

\end{enumerate}
It now suffices to proceed with the proof assuming: $S$ is connected, has no null-homotopic circles, and either $g=0$ and $n \geq 3$, or $g>0$.
This is done in Theorem~\ref{cor:cutintodisc}. 
\end{proof}

\subsection{Ideal triangles} \label{subsec:idealtri}
The aim of this section is to prove Lemma~\ref{lem:flatland}, which is a straightforward generalisation to ideal geodesic triangles of the slim triangle property. To this end, we first prove a few technical Lemmas about $\delta$-hyperbolic spaces.

For a geodesic or geodesic ray $\eta$
we will frequently use the notation $\eta_t=\eta(t)$.

\begin{lemma} \label{lem:tri3} Let $\eta$ and $\eta'$ be geodesic rays such that each one lies in a finite neighbourhood of the other. Then there exist $q, r \geq 0$ such that $\dist(\eta_q, \eta'_r) \leq 2\delta$. Moreover, $q,r$ can be chosen arbitrarily large.
\end{lemma}

\begin{proof}
Let $F>0$ with $\eta\subset N_F(\eta')$ and $ \eta' \subset N_F(\eta)$, and let $\eta'_r$ be such that $\dist(\eta_0, \eta'_r) \leq F$. Choose $\eta_p$ at distance $> F + 2\delta$ %very far 
from $\eta_0$ and $\eta'_r$.
The rectangle in Figure~\ref{fig:rectangles} shows that a point on $\overline{\eta_0 \eta_p}$ at distance more than $F + 2\delta$ from both $\overline{\eta_0\eta'_r}$ and $\overline{\eta_p\eta'_q}$ must be within $2\delta$ of a point on $\overline{\eta'_r\eta'_q}$. 
\end{proof}

\begin{lemma} \label{lem:tri1} Let $\eta, \eta', \eta''$ be as in the statement of Lemma~\ref{lem:flatland}. Then there exist points $\eta_0, \eta'_0, \eta''_0$ at distance $\leq 7\delta$.
\end{lemma}

\begin{proof} By Lemma~\ref{lem:tri3}, there exist $q,q',r,r',s,s'$ such that
\begin{equation*}
  \dist\big(\eta_r,\eta'_{r'}\big) \  \leq  \  2\delta  \hspace{.5in}
 \dist\big(\eta_q,\eta''_{q'}\big) \ \leq  \  2\delta  \hspace{.5in}
 \dist\big(\eta'_s,\eta''_{s'}\big) \ \leq \  2\delta.  
 \end{equation*}
 Consider the geodesic hexagon with sides $\overline{\eta_q\eta''_{q'}}, \overline{\eta''_{q'}\eta''_s}, \overline{\eta''_s\eta'_{s'}},  \overline{\eta'_{s'}\eta'_{r'}}, \overline{\eta'_{r'}\eta_r}, \overline{\eta_r\eta_q}$ and subdivide it by taking geodesics $\beta, \beta', \beta''$ as illustrated in Figure~\ref{fig:hexagon9d}.
 Reparametrising if necessary, let $\eta_0\in \overline{\eta_q\eta_r}$ be the vertex of an intriangle corresponding to the triangle with sides
 $\beta\overline{\eta_q\eta_r}\beta'$.
 The point $\eta_0$ is at distance $\leq \delta$ from a point $x \in \beta$, and $x$ is at most $3\delta$ away from a point $\eta''_0 \in \eta''$.
 Similarly, $\eta_0$ is at distance $\leq \delta$ from a point $y \in \beta'$, and $y$ is at most $4\delta$ away from a point $\eta'_0 \in \eta'$  (the various possibilities are sketched in Figure~\ref{fig:hexagon9d}).
\end{proof}		

\begin{figure}[h!]
\centerline{\includegraphics[scale=0.6]{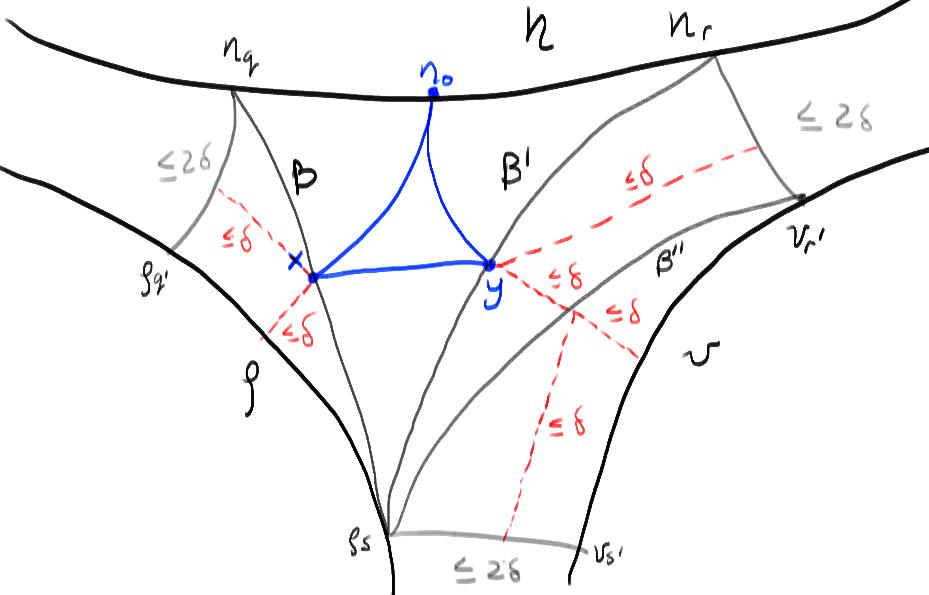}}
\caption{Configuration involved in the proof of Lemma~\ref{lem:tri1}. As the dotted lines indicate, either $y$ is $2\delta$ away from $\eta'$, or $y$ is $2\delta$ away from $\overline{\eta''_s\eta'_{s'}}$, in which case $y$ is $4\delta$ away from $\eta'$, or  $y$ is $2\delta$ away from $\overline{\eta'_{r'}\eta_r}$, in which case $y$ is $4\delta$ away from $\eta'$.}
\label{fig:hexagon9d}
\end{figure}

\begin{lemma} \label{lem:tri2} Let $\eta, \eta'$ be geodesic rays lying in finite neighbourhoods of each other in a $\delta$-hyperbolic geodesic metric space.
If $\dist(\eta_0,\eta'_0) \leq K$ then $\dist(\eta_t,\eta'_t) \leq K'(\delta)=K'$ for all $t\geq 0$.
\end{lemma}

\begin{proof}
%We will bound $\dist(\eta_m,\eta'_m)$. 
By Lemma~\ref{lem:tri3}, there are $q,r > m$ with $\dist(\eta_q, \eta'_r)\leq 2\delta$. 
Let $Q$ be the quadrilateral with sides $\overline{\eta_0\eta'_0}$, $\overline{\eta_q\eta'_r}$,  $\overline{\eta_0\eta_q}$ and  $\overline{\eta'_0\eta'_r}$.
We claim that  $d(\eta_n, \eta_m)\leq K +2\delta$ for $\eta_n, \eta_m \in \overline{\eta_0\eta_r}$, similarly, $d(\eta'_n, \eta'_m)\leq K +2\delta$ for $\eta'_n, \eta'_m \in \overline{\eta'_0\eta'_r}$. Indeed: $m\leq K + |\eta_n|+2\delta=K +n+2\delta$, so $m-n\leq K+2\delta$ and $n\leq K + |\eta_m|+2\delta=K +m+2\delta$, so $|n-m|\leq K+2\delta$. Hence $|m-n|=d(\eta_n, \eta_m)\leq K +2\delta$ and similarly for $d(\eta'_n, \eta'_m)$.
Now we will bound $\dist(\eta_m,\eta'_m)$. There are 2 cases. See Figure~\ref{fig:rectangles}.
\begin{enumerate}
\item If there exists $n$ with $d(\eta_m,\eta'_n)\leq 2\delta$, then $\dist(\eta_m,\eta'_m)\leq \dist(\eta_m,\eta'_n)+\dist(\eta'_n,\eta'_m)\leq 2\delta+K+2\delta$. The same holds by a symmetric argument if there exists $n$ with $d(\eta'_m,\eta_n)\leq 2\delta$.
\item Otherwise, both $\eta_m$ and $\eta'_m$ are within distance $2\delta$ of $\overline{\eta_0\eta'_0}$. Since $|\overline{\eta_0\eta'_0}|=K$, it follows that $\dist(\eta_m,\eta'_m)\leq 2\delta +K + 2\delta$.
\end{enumerate}
Either way, $d(\eta_, \eta'_m)\leq K +4\delta:=K'.$
\end{proof}

\begin{figure}[h!]
\centerline{\includegraphics[scale=0.6]{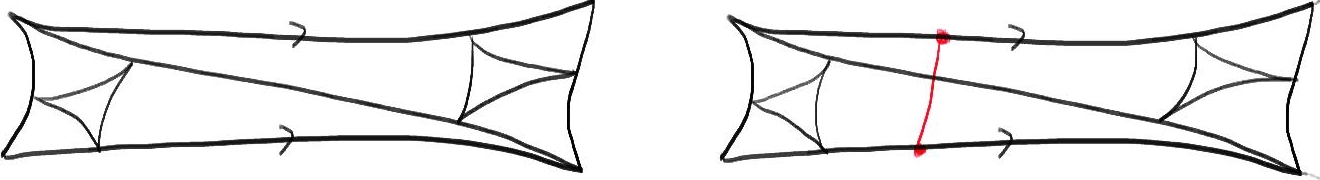}}
\caption{Rectangles used in the proofs of Lemma~\ref{lem:tri2} and Lemma~\ref{lem:tri3}.}
\label{fig:rectangles}
\end{figure}

\begin{lem}[$\delta$-ideal triangles]\label{lem:flatland}
Let $\widetilde X$ be a $\delta$-hyperbolic 
metric space.
Let $\eta, \eta',\eta''$ be bi-infinite geodesics in $\widetilde X$, that form an ideal triangle in the sense that
there are three points $\{a,b,c\}\subset \boundary X$ with
 $\boundary \eta = \{b,c\}$ and $\boundary \eta'=\{a,c\}$ and $\boundary \eta'' = \{a,b\}$.
 
 There exist isometric reparametrisations  of $\eta$, $\eta'$, and $\eta''$   with:
 \begin{equation*}
 \begin{aligned}
 \dist\big(\eta_t,\eta'_t\big)  &\leq & K'  \ \ & \ \text{ for } \ \ t\geq 0\\
 \dist\big(\eta_{-t},\eta''_{-t}\big) &\leq & K' \ \ & \ \text{ for } \ \  t\geq 0\\ 
 \dist\big(\eta'_{-t},\eta''_t\big) &\leq & K'  \ \ & \ \text{ for }  \ \ t \geq 0\\
\end{aligned}
 \end{equation*}
\end{lem}

The sides of a $\delta$-ideal triangle are called \emph{verticals}.

\begin{defn}\label{defn:intriangle}
Let 
%In the context of Lemma~\ref{lem:flatland}, regard
 $\{\eta_0, \eta'_0, \eta''_0\}$ be as in Lemma~\ref{lem:flatland}, 
 an %as vertices of an 
 \emph{intriangle} is the set $\Lambda$ determined by three geodesics arcs $s, s', s''$ --called \emph{sides}-- with endpoints  $\{\eta_0, \eta'_0, \eta''_0\}$. 
See Figure~\ref{fig:insize}.
%We emphasize that   diameter$\{\eta_0, \eta'_0, \eta''_0\} \leq 7\delta$.
% by  Lemma~\ref{lem:flatland} %and from Lemma~\ref{} 
\end{defn}
\begin{proof}[Proof of Lemma~\ref{lem:flatland}]
%The proof relies on three simple Lemmas  presented below.
By Lemma~\ref{lem:tri1}, there are points $\eta_0 \in \eta, \eta'_0 \in \eta', \eta''_0 \in \eta''$ at distance $\leq 7\delta$.
Hence by Lemma~\ref{lem:tri2} with $K=7\delta$ we have:
 \begin{equation*}
 \begin{aligned}
 \dist\big(\eta_t,\eta'_t\big)  &\leq & K'  \\
 \dist\big(\eta_{-t},\eta''_{-t}\big) &\leq & K' \\ 
 \dist\big(\eta'_{-t},\eta''_t\big) &\leq & K'   \\ 
 \end{aligned}\qedhere
 \end{equation*}
\end{proof}

\begin{figure}%[H]
\centerline{\includegraphics[scale=0.64]{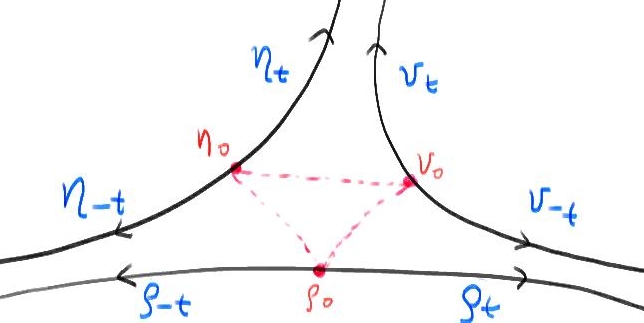}}
\caption{The configuration described in Lemma~\ref{lem:flatland}.}
\label{fig:insize} 
\end{figure}

\subsection{The ideal retriangulation} \label{subsec:retriangulate}

\begin{notation}
To simplify notation we let $\chi=\chi(S)$. This is non-ambiguous: we will modify and simplify $S$ in various ways --but Euler characteristic remains invariant under all of these modifications.
\end{notation}

\begin{construction}[$\delta$-ideal triangulation]\label{cnst:idealtri} 
Let $S \rightarrow X$ be a genus~$g$ diagram with $n$ boundary components $C_1, \dots, C_n$
% that map essentially to $X$ and that do not represent torsion elements in $\pi_1X$, 
such that each $C_i \rightarrow X$ is essential.%whose universal covers embed  as uniform quasigeodesics $\widetilde{C}_i \subset \widetilde{X}$.
 We shall construct an associated ``infinite diagram" 
 $S^\delta\rightarrow X$ whose universal cover is built from geodesic $\delta$-ideal triangles $\{\widetilde{\triangle^\delta}_j\}$.

We now recall the compact genus~$g$ surface $\bar{S}$ that deformation retracts to $S \subset \bar{S}$, and this deformation retraction sends each boundary circle $\bar{C_i}$ of $\bar{S}$ to the corresponding $C_i$. We regard $\bar{S}$ as having a complete hyperbolic metric. The interior of $\bar{S}$ can be decomposed into geodesic ideal triangles $\{\triangle_j\}$  meeting in pairs, and each ray of each geodesic side limits to some $\bar{C_i}$. Indeed, it suffices to find a triangulation of a closed genus~$g$ surface with exactly $n$ vertices, replace the triangles by ideal triangles, and glue corresponding sides with a displacement, so that the sum of all displacements around a given ideal vertex corresponds to the length of the desired boundary circle. 

%Moreover, these ideal triangles can be taken to be geodesic ideal triangles by endowing $S$ with a complete hyperbolic metric. 
An Euler characteristic calculation %(viewing the boundary circles as vertices of a triangulation of $S^2$)
shows that $\#\{\triangle_j\}=2|\chi|=2(2g+n-2)$ and so there are $3(2g+n-2)$ sides in the ideal triangulation.

Composing the deformation retraction $\bar{S} \rightarrow S$ with the map $S \rightarrow X$ yields a map $\widetilde{\bar{S}} \rightarrow \widetilde{X}$. Each $\widetilde{C_i}$  maps to a periodic path in $\widetilde{X^1}$ which is a quasigeodesic and thus each $\widetilde{\bar{C_i}}$ maps to that quasigeodesic.
Each side $s$ maps to a quasigeodesic $\tilde s$ in $\widetilde{X}$ under the induced map $\widetilde{S} \rightarrow \widetilde{X}$. 

If the endpoints of $\tilde s$ coincide in $\partial \widetilde{X}$, then $\tilde s$ is null-homotopic, and so $s$ is null-homotopic. This contradicts the assumption that $S$ has no null-homotopic circles.
Therefore we may assume that the endpoints of $\tilde s$  in $\partial \widetilde{X}$ are distinct. This ensures the existence of a geodesic $\tilde{\gamma}$ with the same endpoints as $\tilde s$ in $\partial \widetilde{X}$.  

Choose a geodesic $\tilde{\gamma}$ with the same endpoints as $\tilde s$ in $\partial \widetilde{X}$.  
For each lift $\widetilde \triangle_j$ of an ideal triangle in $\{\triangle_j\}$, a choice of a triple of geodesics $(\tilde{\gamma}_1, \tilde{\gamma}_2, \tilde{\gamma}_3)$ defines a \emph{$\delta$-ideal triangle} $\widetilde{\triangle^\delta}_j$ in $\widetilde X$ as in Lemma~\ref{lem:flatland}. Repeating this procedure for each lift of each $\triangle_j$ yields a set of $\delta$-ideal triangles $\{\widetilde\triangle^\delta_j\}$ associated to $\{\widetilde \triangle_j\}$.  
\end{construction}

\begin{figure}[h!]
\centerline{\includegraphics[scale=0.6]{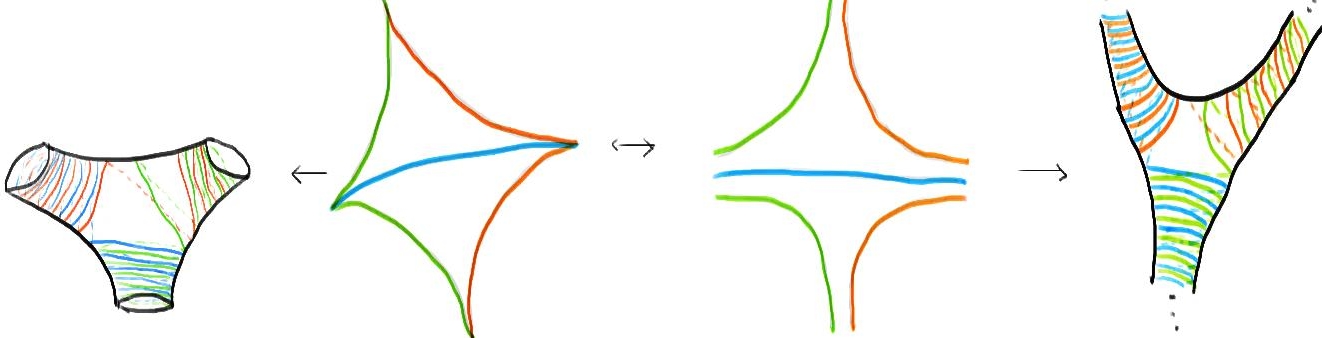}}
\caption{Heuristics of Construction~\ref{cnst:idealtri} for a ``pair of pants" genus~0 diagram.}
\label{fig:pants}
\end{figure}

\begin{remark}[Choices] 
The construction described above is not canonical: the topological ideal triangulation of $S$ is not unique, and therefore neither is the induced $\delta$-ideal triangulation, nor the intriangles. For the remainder of the paper, we assume that one such choice has been made and is kept throughout.
\end{remark}

\begin{construction}[$\delta$-ideal triangles in $\widetilde X$]\label{const:infinite}
Let $\Lambda$ be a disc diagram whose sides are geodesics with endpoints at $v_0, v_1, v_2$. Attach an infinite rectangular disc diagram $D_j$ at each side, where $D_j$ is obtained via the following procedure.

Let $r=e_1e_2\cdots$ and $r'=e_1'e_2'\cdots$ be geodesic rays representing the same points in $\partial \widetilde X$. Let $\mu_0$ be a geodesic joining their initial points. For each $i$ let $\mu_i$ be a geodesic joining $r(i)$ and $r'(i)$. Let $R_i$ be a disc diagram for $\mu_ie_i\mu_{i+1}^{-1}e_{i+1}^{-1}$. Let $D=\cup_i R_i$. Then $\Lambda \cup_j D_j \rightarrow \widetilde X$ maps to a $\delta$-ideal triangle in $\widetilde X$.
\end{construction}

\begin{definition}
Choose basepoints $p_{ij} \in s_{ij} \subset \triangle_j$ for each $i \in \{1,2,3\}$ and $j\in \{1, \dots, 2(n-2)\}$. The map  $\cup_j \widetilde{\triangle}_j \rightarrow \cup_j \widetilde{\triangle^\delta}_j$ that assigns a geodesic $\tilde \gamma$ to each $s$, sends the $p_{ij}$ to basepoints $\hat p_{ij}$.  Therefore the triangles of
 $\{\widetilde{\triangle^\delta}_j\}$ can be glued together unambiguously to form $S^\delta = (\sqcup \widetilde{\triangle^\delta}_j) / \sim$.
 
Explicitly, if $\{\gamma_k\}_{k \in K}$ are the various verticals and $\phi_{k0}:\gamma_k \rightarrow  \triangle^\delta_{i(k0)}$ and $\phi_{k1}:\gamma_k \rightarrow \triangle^\delta_{i(k1)}$ are the two identifications between $\gamma_k$ and the sides of $\delta$-ideal triangles, then we have:
\begin{equation}\label{eq:gluingtriangles}
S^\delta = (\sqcup \widetilde{\triangle^\delta}_j) \ / \ \{\phi_{k0}(x)\sim \phi_{k1}(x) \ : \ k \in K \text{ and }  x \in \gamma_k\}
\end{equation}

\end{definition}

\begin{remark} There is a niggling difference between a hyperbolic ideal triangle and its corresponding $\delta$-ideal triangle. While the former is homeomorphic to a disc minus 3 points, the latter could be singular and even contain infinitely many cut points.
 In the extreme case where $\widetilde X$ is a tree, every $\delta$-ideal triangle is an infinite tripod.

Because of these cut points, it is a priori possible that the quotient space $S^\delta = (\sqcup \widetilde{\triangle^\delta}_j) / \sim$ is not homeomorphic to a surface. See the right of Figure~\ref{fig:buffer}.  

This possibility is not directly covered by the method in Section~\ref{subsec:horizontalpt}, and it will be convenient to adopt a workaround. For this purpose we ``buffer" $S^\delta$ in a way that avoids singularities and actually maintains the homeomorphism type of the interior of $S$.
\end{remark}

\newcommand{\D}{O}

\begin{construction}[Buffer]\label{cnst:buffer}

For each vertical $\gamma_k$, define the \emph{buffer}:

\begin{equation*}
\gamma_k\times I \ \cong \ \big ((\gamma_k\times[0,\frac{1}{2}])\sqcup ([\frac{1}{2},1]\times \gamma_k)\big ) \ \big / \ \big \{(x,\frac{1}{2})\sim (\frac{1}{2},x) : \forall x \in \gamma_k \big \}
\end{equation*}

Thus, the buffer of $\gamma_k$ is essentially the product $\gamma_k\times [0,1]$ whose cells have been subdivided by adding $\gamma_k\times\{\frac{1}{2}\}$ to the 1-skeleton.

We create a new diagram by taking a buffer for each $k \in K$, and identifying its ``left" and ``right" sides with the sides of $\delta$-ideal triangles that $\gamma_k$ maps to in $S^\delta$.

Explicitly, we replace Equation~\eqref{eq:gluingtriangles} with the following:
\begin{equation*}
S^\times = (\bigsqcup_j \widetilde{\triangle^\delta}_j \sqcup \bigsqcup_k (\gamma_k \times I)) \ / \ \{\phi_{k0}(x)\sim (x,0), \ \phi_{k1}(x)\sim (x,1) \ : \ k \in K \text{ and }  x \in \gamma_k\}
\end{equation*}

The diagram $S^\times$ is non-singular, and decomposes as a union of ideal triangles and buffers. 

By relaxing the condition of diagrams in $X$ to allow the map $S^\times \rightarrow X$ to be cellular instead of combinatorial, we obtain a  genus~$g$ surface with $n$ punctures in $X$ via the composition $S^\times \rightarrow S^\delta \rightarrow X$, where the map $S^\times \rightarrow S^\delta $ is the cellular map that collapses each buffer via projection to the first factor to a vertical $\gamma$ in $S^\delta$.

While it might seem counterproductive to replace a diagram with one that has larger area, the reader should keep in mind that $S^\times$ is only an accessory, and may have  little relation with the diagram finally obtained in Section~\ref{subsec:horizontalpt}. 
\end{construction}

\begin{figure}[h!]
\centerline{\includegraphics[scale=0.45]{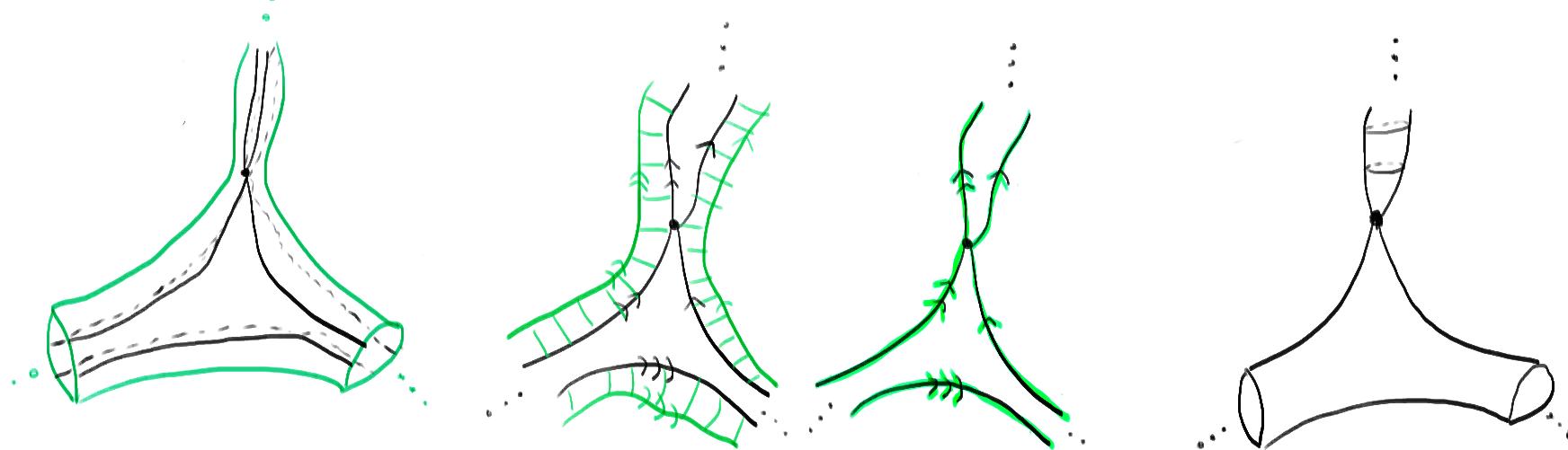}}
\caption{At the left is the ``buffered" diagram with the buffer drawn in green. At the centre are the ideal triangles. At the right is the original singular diagram.}
\label{fig:buffer}
\end{figure}

\section{Jumps, horizontal paths, and bands}\label{subsec:horizontalpt}

\subsection{The graph of spaces decomposition}

Let $\Lambda_j^\times \subset S^\times$ be the \emph{buffered intriangle} obtained from $\Lambda_j$ by attaching the three half edges whose initial vertices are the vertices of $\Lambda_j$ and whose terminal vertices lie on $\{\gamma_i\}$. 
See Figure~\ref{fig:bufferednons}. Let $D_j^i$ be the rectangular strips attached to $\Lambda^\times$ via Construction~\ref{const:infinite}, where $i \in \{1,2,3\}$. 
Then $\cup_{i,j}D_j^i$ decomposes as a graph of spaces, where each vertical $\gamma_i\times \{\frac{1}{2}\}$ is a vertex space and each $D_j^i$ is an edge space.  
  $S^\times$ can then be recovered by adding the buffered intriangles to $\cup_{i,j}D_j^i$, and so  $\Gamma$ is naturally associated with $S^\times$, despite only corresponding to a decomposition of a subspace of $S^\times$.

Since the $\{\gamma\times \{\frac{1}{2}\}\}$ are in correspondence with the buffers, $\Gamma$ has a vertex $v_\gamma$ for each $\gamma\times \{\frac{1}{2}\}$ and an edge $e_J$ for each parallelism class of jumps $[J]$. Hence $\#V(\Gamma)=3|\chi|$ and $\#E(\Gamma)=6|\chi|$.

\begin{remark}
When all ideal triangles of $S^\delta$ are non-singular, the graph of spaces decomposition can be defined analogously without having to utilise $S^\times$. However, the above considerations are necessary when there are singularities -- for example, when the ideal triangles are tripods as in Figure~\ref{fig:bufferednons}. 
\end{remark}

\begin{figure}[h!]
\centerline{\includegraphics[scale=0.56]{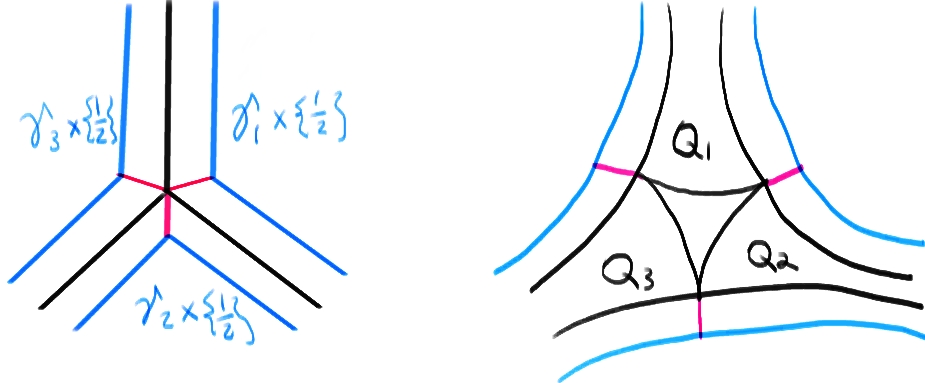}}
\caption{A buffered tripod, and a buffered ideal triangle with the $Q_i$ described in the graph of spaces decomposition.}
\label{fig:bufferednons}
\end{figure}

\begin{definition}
A \emph{trajectory} is a path $T \rightarrow \Gamma$. 
It is \emph{semi-embedded} if it traverses every edge at most once in each direction. 
\end{definition}

A finite graph has finitely many semi-embedded trajectories. In fact, a coarse bound is straightforward:

\begin{lemma}\label{lem:numberofsemis} $\#\{T \rightarrow \Gamma : T \text{ is semi-embedded}\}\leq (12|\chi|+1)!$.
\end{lemma}

\begin{proof}

Let $\mathcal{O}$ be the set of oriented edges of $\Gamma$, so  $\#\mathcal{O}=2\#E(\Gamma)=12|\chi|$. No oriented edge can occur more than once in a semi-embedded trajectory, hence there are $12|\chi|$ choices for the first edge of such a trajectory, $12|\chi|-1$ choices for the second edge, and so on.
\end{proof}

\begin{notation} In what follows all paths and diagrams map to $S^\times$. Hence, to simplify notation, we write $\gamma$ in place of $\gamma \times \{\frac{1}{2}\}$ and $\Lambda$ in place of $\Lambda^\times$.
\end{notation}

\subsection{Jumps and horizontal paths}\label{subsec:hp}

We now describe ``horizontal paths" in $S^\times$ that serve as combinatorial analogues to a ``horocyclic flow". These will allow us to control the length of certain geodesic segments in $S^\delta$ and afterwards also in the ``trimmed" genus~$g$ diagram $\Sigma$. As geodesics are not unique, the definition is involved.%it requires a bit of extra care to define these horizontal paths rigorously.

\begin{definition}[Jumps]\label{def:jump}
A \emph{jump} is a geodesic arc $J$ whose corresponding trajectory $T_J \rightarrow \Gamma$ has length $1$.
\end{definition}

Two jumps $J,J'$ are \emph{opposite} if 
they lie in the same buffered ideal triangle,
and the terminal vertex of $J$  is the initial vertex of $J'$
(so the initial of $J$ is the terminal of $J'$). 
 Two jumps $J,J'$ are \emph{parallel} if 
their corresponding trajectories $T_J$ and $T_{J'}$ are equal. 
This is equivalent to saying that $J$ and $J'$ lie in the same ideal triangle and their initial vertices lie in the same vertical $\gamma$ but are not separated by either vertex of an intriangle on $\gamma$. 
We refer to  Figure~\ref{fig:lrpaths}.

\begin{definition}[Horizontal Paths]\label{def:flow}

A \emph{ horizontal  path} $h$ is a path in $S^\times$ that is a concatenation of jumps
such that:
Firstly, $h$ has no \emph{backtrack} 
consisting of a pair of consecutive jumps that are opposite.
Secondly, no two jumps of $h$ have the same initial point,
and no two jumps of $h$ have the same terminal point.

The \emph{jump length} of a horizontal path $J_1\cdots J_N$  equals the number $N$ of jumps. 
\end{definition}

\begin{figure}[h!]
\centerline{\includegraphics[scale=0.54]{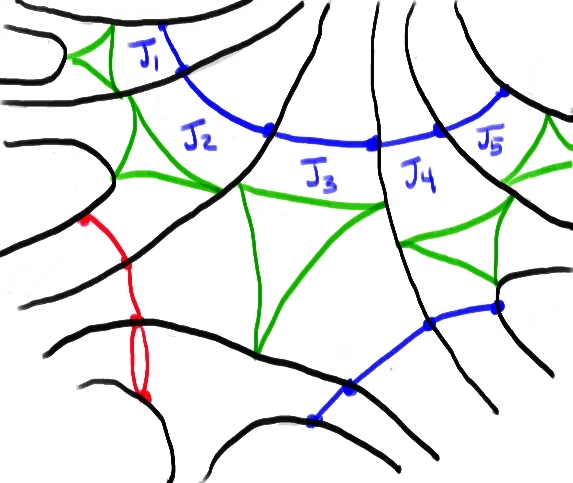}}
\caption{Two horizontal paths in the universal cover, and a path with a backtrack (in red).}
\label{fig:lrpaths}
\end{figure}

Let $\iota(h)$ and $\tau(h)$ denote the sets of initial and terminal points of jumps of $h$, respectively.
The horizontal paths $h, g$ are  \emph{equivalent} if 
 $\iota(h) = \iota(g)$ and $\tau(h) = \tau(g)$.

\subsection{Trimming the surface}

\begin{definition}

A \emph{$\boundary$-return} is a horizontal path 
 $h=J_1 J_2\cdots J_n$ such that:
\begin{enumerate}
  \item $T_h$ is closed and semi-embedded 
  \item there is a vertical $\gamma_i$ 
   and a $\lambda_i\subset \gamma_i$ with the same endpoints as $h$
   \item the concatenation $h\lambda_i$ is a cycle that separates $S^\delta$ into two  components 
   \item all intriangles of $S^\delta$ lie in the same component.
\end{enumerate} 
An \emph{augmented $\boundary$-return} is a concatenation $h\lambda_i$ as above.
\end{definition}

Observe that $S^\times$ contains a surface $\bar{S}^\times$ such that 
$S^\times - \bar{S}^\times$ consists of $n$ open cylinders $\sqcup  \mathcal{L}_i$,
which we refer to as the ``cusps'' of $ S^\times$. 

\begin{lemma}\label{lem:partialexists}
$\partial$-returns exist.
\end{lemma}

\begin{proof}

For each edge space $D_i$ travelling into a fixed cusp $L$, choose a jump $J_i$ in the parallelism class of $D_i$. Furthermore, choose the $J_i$ so that they are concatenable (i.e., $J_i(1)=J_{i+1}(0)$ for each $i \in I$). This can always be attained by ``pushing" the $J_i$ away from $\Lambda_i$.

Let $h=J_1\cdots J_k$ be such a concatenation satisfying $J_1(0), J_k(1) \in \gamma$ for some vertical $\gamma$, and where $k$ is the minimal jump length $k$ for which such a horizontal path exists. Then $T_h$ is semi-embedded since $k$ is minimal, and $T_h$ is closed since $J_1(0), J_k(1) \in \gamma$.
Moreover, all intriangles of $S^\times$ lie on the same component of $S^\times - h$ since the  $D_i$  all travel into $L$.
\end{proof}

\begin{remark}\label{rem:semi} 
Let $T_h \rightarrow \Gamma$ be a semi-embedded trajectory, then $|T_h|\leq 12|\chi|$. Indeed,  a semi-embedded trajectory has length at most twice the number of edges, i.e., $2 \cdot 6|\chi|$. 
\end{remark}

Since $\partial$-returns project to semi-embedded trajectories, we conclude that:

\begin{corollary}\label{lem:oneside} Let $h$ be a $\partial$-return, then $|h|\leq 12|\chi|$. 
\end{corollary}

We use the notation $h_x$ below
to indicate a horizontal path having $x$ as an initial or terminal vertex of some jump.

\begin{construction}[Trimming $S^\times$]\label{ct:trim}
We now describe how to obtain $\Sigma$ from $S^\times$. 

For each $\mathcal{L} \in \{\mathcal{L}_i\}$ choose a point $x$ in a vertical $\gamma$ such that there is an augmented  $\partial$-return $h_{x}\lambda$  where $\lambda$ has endpoints $x,y$, and satisfying the following:

\begin{enumerate}

\item \label{cond1}$x$ is a vertex of an intriangle in $S^\times$,
\item \label{cond2} $h_x\lambda$ separates $\cup_\ell \Lambda_\ell-\{x\}$ from $\mathcal{L}$, 
 \end{enumerate}

In view of Lemma~\ref{lem:partialexists}, condition~\ref{cond2} can always be attained. Condition~\ref{cond1} can be achieved by ``pushing down" a $\partial$-return that satisfies Condition~\ref{cond2} until it intersects an intriangle at a vertex.

Let $\{h_{x_1}\lambda_1, \ldots, h_{x_n}\lambda_n\}$ be the augmented $\partial$-returns obtained above. 
Note that each $h_{x_i}\lambda_i$ bounds an infinite annulus in $ S^\times$.
Let $\Sigma$ be obtained from $S^\times$ by removing the annuli corresponding to each of the $h_{x_i}\lambda_i$. Then $\Sigma \hookrightarrow S^\times$ is a compact genus~$g$ surface with $n$ boundary components $\mathcal{C}_1=h_{x_1} \lambda_1, \dots, \mathcal{C}_n= h_{x_n} \lambda_n$. 
\end{construction}

To relate the length of $\partial \Sigma$ to the length of $\partial S$, %the $\mathcal{C}_i$ to the length of the $C_i$, 
we will need the following:

\begin{figure}[h!]
\centerline{\includegraphics[scale=0.52]{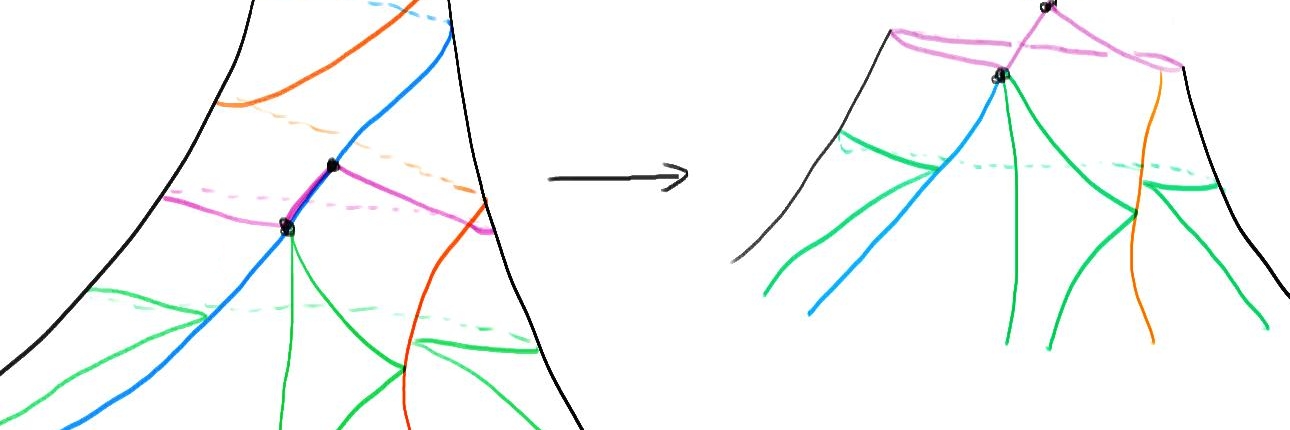}}
\caption{Trimming along an augmented $\boundary$-return to obtain $\Sigma$.} 
\label{fig:trim}
\end{figure}

\begin{lemma}\label{lem:quasig}
Let $h\lambda$ be an augmented  $\boundary$-return in $S^\times$. Then the universal cover $\widetilde{h\lambda}$ maps to an $(a,b)$-quasigeodesic in $\widetilde X$, where $a,b$ depend only on $g$, $n$, and $\delta$.
\end{lemma}

\begin{remark} We emphasize that Lemma~\ref{lem:quasig} does not require that $S^\times$ be orientable. The proof instead hinges on the fact that no surface contains a Möbius strip homotopic into the boundary of the surface. 
\end{remark}

\begin{proof}
We will apply Theorem~\ref{thm:loco}, where the lifts of $h$ play the roles of the various $\sigma_i$ and the lifts of $\lambda$ play the roles of the various $\lambda_i$.
First we show there is $L'>0$ such that for any two consecutive lifts $\tilde \lambda$ and $\tilde \lambda$, the intersection $\tilde \lambda \cap \mathcal{N}_{3\delta}(\tilde \lambda')$ has diameter at most $L'$. Let $\tilde \gamma_j, \tilde \gamma'_j$ be the lifts of $\gamma_j$ with $\tilde \lambda \subset \tilde \gamma_j$ and $\tilde \lambda' \subset \tilde \gamma'_j$.
 Let $\bar{h}$ be the lift of $h$ connecting $\tilde \lambda$ to $\tilde \lambda'$.
 By Theorem~\ref{thm:cdiv}, it suffices to show that the subrays $\bar{\gamma}_j:[0, \infty) \rightarrow \Sigma$ and $\bar{\gamma}_j' :[0, \infty) \rightarrow \Sigma$  of $\tilde{\gamma}_j$ and $\tilde{\gamma}_j'$ having an endpoint on $\bar h$ and containing respectively $\tilde \lambda$ and $\tilde \lambda'$, do not represent the same point on  $\partial \widetilde X$.
 
Notice that the lift $\bar{h}_0$ of the horizontal path containing $h$ is two-sided in the sense that $\bar{h}_0$ separates $\widetilde S^\times$, and  $\bar{\gamma}_j$ and $\bar{\gamma}_j'$ lie on opposite sides of $\bar{h}_0$. 
Arguing by contradiction, suppose $\bar{\gamma}_j$ and $\bar{\gamma}_j'$ represent the same point on $\partial \widetilde X$.
Since $\bar{\gamma}_j$ and $\bar{\gamma}_j'$ both have an endpoint on $\bar{h}$, it follows that $\tilde \lambda(1)$ and $\tilde \lambda'(0)$ project to the same point in the quotient,
 and a neighbourhood of $\tilde h$ connecting the lifts $\tilde \gamma_j$ and $\tilde \gamma'_j$ as in Figure~\ref{fig:mobius}
 produces a Möbius strip in $S^\times$ that is homotopic into the boundary, which is impossible. 
Hence, $\bar{\gamma_j}$ and $\bar{\gamma_j}'$ have different endpoints on $\partial \widetilde X$. 

Since $\bar{\gamma_j}(1)=h(0)$ and $h(1)=\bar{\gamma_j}'(0)$, and $\dist(\bar{\gamma_j}(0), \bar{\gamma_j}'(0))=|h| \leq 12|\chi|$ by Corollary~\ref{lem:oneside}, it follows from Corollary~\ref{cor:closediverge} that there exists $L'>0$ for which $\bar{\gamma_j}, \bar{\gamma_j}'$ do not lie in the $3\delta$-neighbourhood of each other.

Since $|h| \leq 12|\chi|$, choosing $L''>12|\chi|$ ensures that $\diam\big(\tilde \lambda \cap \mathcal{N}_{3\delta}(\tilde h\tilde \lambda')\big )\leq L''$, so letting $L=max\{L', L''\}$ yields the desired result, provided that $\frac{1}{2}|\tilde{\lambda}| \geq 6(L+\delta)$. If $\frac{1}{2}|\tilde{\lambda}| <6(L+\delta)$ then $\frac{1}{2}|\lambda|<6(L+\delta)$, and so $\tilde h\lambda$ is still a uniform quasigeodesic, since it projects to an essential path in $\Sigma$ of length uniformly bounded by $6(L+\delta) + 12|\chi|$ and there are only finitely many such combinatorial paths.
\end{proof}

\begin{figure}[h!]
\centerline{\includegraphics[scale=0.5]{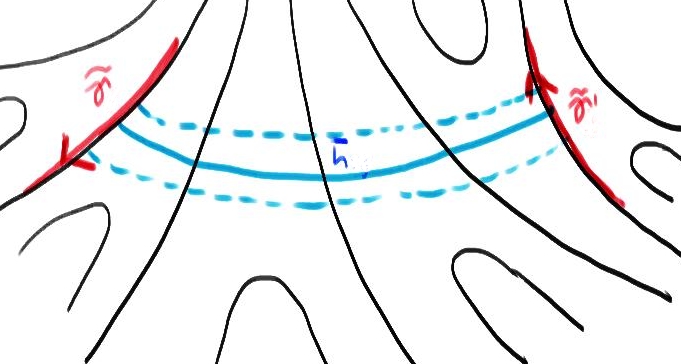}}
\caption{Möbius strip.}
\label{fig:mobius}
\end{figure}

\begin{corollary}
For every boundary component $\mathcal{C}_j$ of $\Sigma$, the universal cover $\widetilde{\mathcal{C}}_j$ maps to an $(a,b)$-quasigeodesic in $\widetilde X$, where $a,b$ depend only on $g$, $n$, and $\delta$.
\end{corollary}

\begin{corollary}\label{cor:lengthbdry}
In the setting of Construction~\ref{ct:trim}, each boundary component $\mathcal{C}_i$ of $\Sigma$ is homotopic in $X$ to a boundary component $C_i$ of $S$, and $\partial \Sigma=V \cup H$ where $V=\cup \lambda_j$ and $H=\cup_j h_j$.
Moreover, there is a constant $\mathcal{K}>0$ with $|\mathcal{C}_i|\leq \mathcal{K}|C_i|$. See Figure~\ref{fig:trim}.
\end{corollary}

\begin{corollary}\label{cor:htpysurf} There is a genus~$g$ diagram $\Sigma'$ homotopic to $\Sigma$ and such that $\partial \Sigma'=\partial S$. Moreover, $\area(\Sigma')\leq \area(\Sigma)+ nM|\partial \Sigma|$. 
\end{corollary}

\begin{proof}
This follows from Proposition~\ref{prop:BHconj} and Corollary~\ref{cor:lengthbdry}.
\end{proof}

\begin{remark}\label{rmk:restrict}
Henceforth we restrict our attention to surface diagrams $\Sigma$ obtained via Construction~\ref{ct:trim}. By Corollary~\ref{cor:htpysurf}  it suffices to bound $\area (\Sigma)$ to prove Theorem~\ref{thm:main}.
\end{remark}

\subsection{Combinatorics of bands}

The horizontal paths $h$ and $h'$ are \emph{parallel}
if 
 their sequences of initial points of jumps
$p_1,\ldots, p_m$ and $p'_1,\ldots, p'_m$ have the same length
and moreover,
$p_i, p'_i$ lie on the same %vertical
vertical $\gamma$  for each $i$,
but are not separated by either vertex of an intriangle on $\gamma$.
Equivalent horizontal paths are obviously parallel,
and parallelism is an equivalence relation.

\begin{definition} \label{def:bands}
Let $h,h'$ be parallel horizontal path. There is a ``rectangular" disc diagram $B \rightarrow \Sigma$ with boundary $\partial B=h^{-1}\iota_{_P}^{-1}h'\tau_{_P}$, where 
$\iota_{_P}$ is the %vertical
subgeodesic of $\gamma$ bounded by $x,x'$, where
$\tau_{_P}$ is the subgeodesic of  $\gamma'$ bounded by the terminal points of $h,h'$, 
and where the orientations of $\iota_{_P}, \tau_{_P}$ are chosen to ensure concatenability. We refer to $B$ as a \emph{band}.

The \emph{thickness} of $B$ is $\dist_\gamma(x,x')$.
We say $B$ is \emph{bounded} by $h$ and $h'$, and let $\iota_{_B}=\iota_{_P}$ and $\tau_{_B}=\tau_{_P}$.
%We shall always assume that $h$ and $h'$ do not cross within $\Sigma$ though it is possible for them to touch at vertices or edges.   

The \emph{trajectory} of a band is the path $T_B\rightarrow \Gamma $ that equals the composition $h \rightarrow \Sigma \rightarrow \Gamma$.
\end{definition}

Of paramount importance are  bands that are \emph{maximal} with respect to their trajectory. They have the property that they do not properly factor through another band. %We shall also restrict consideration to bands bounded by horizontal paths $h,h'$ that do not cross in $\Sigma$. However it is possible for them to overlap along subpaths.  \begin{com} maybe permute stuff?? \end{com}
A maximal band $B$ is uniquely determined by a trajectory $T_B$ in $\Gamma$ because of the following: 

\begin{lemma}\label{lem:trajectory} Let $h,h'$ be horizontal paths projecting to the same trajectory in $\Gamma$, then $h$ and $h'$ are parallel.
\end{lemma}

\begin{proof}
Let $h,h'$ project to  $T \rightarrow \Gamma$, let $h=J_1\cdots J_m$ and $h'=J'_1\cdots J'_{m'}$.
Note that $m=m'$, since both horizontal paths must have as many jumps as the length of $T$. 
Finally, $J_i$ is parallel to $J'_i$ for  $1\leq i \leq m$.
\end{proof}

\begin{definition}\label{def:band-a}

A band is \emph{semi-embedded} if its trajectory is semi-embedded.

A maximal band $B$ is \emph{annular} if $T_B$ is semi-embedded and has a closed trajectory whose first and last edges do not lie in the same buffered ideal triangle.

An annular band $B$ is a \emph{snail} if  $\iota_{_B}\cap \tau_{_B}= \emptyset$, otherwise $B$
  is a \emph{spiral}. A \emph{cylinder} is a spiral having $\iota_{_B}=\tau_{_B}$. See Figure~\ref{fig:abands}.
  
   A spiral $B$ is \emph{twisted} if %is an annular band with $\iota_{_B}\cap\tau_{_B}$ not empty or a singleton, and whose bounding horizontal paths (or path in the case $\iota_{_B}=\tau_{_B}$), concatenated with the subpaths of $\iota_{_B} \cup \tau_{_B}$ connecting them, bound 
  its image in $\Sigma$ contains a Möbius strip. A twisted spiral having  $\iota_{_B}=\tau_{_B}$ is a \emph{twisted cylinder}. Twisted spirals can only arise in nonorientable surface diagrams. We use twisted spirals in Section~\ref{sec:nono}.

The \emph{boundary} of an annular band $A$ consists of the two concatenations $hv$ and $h'v'$  where $v$ is the subpath of $\gamma$ whose endpoints are the endpoints of $h$ and $v'$ is the subpath of $\gamma$ whose endpoints are the endpoints of $h'$.  
\end{definition}

\begin{warning}
We caution that the definition of boundary given above for annular bands does not coincide with the topological boundary of a band.
\end{warning}

\begin{remark}\label{rmk:separate} Let $A$ be an annular band, whose associated trajectory is $e_1\cdots e_k$. Let $J_1$ and $J_k$ be the corresponding jumps in $\Sigma$ and let $\gamma$ be the vertical containing $J_1(0)$ and $J_k(1)$.  Then $J_1$ and $J_k$ lie in opposite sides of $\gamma$, since $J_1$ and $J_k$ lie in distinct ideal triangles and $T_A$ is closed. 
\end{remark}

\begin{figure}[h!]
\includegraphics[scale=0.575]{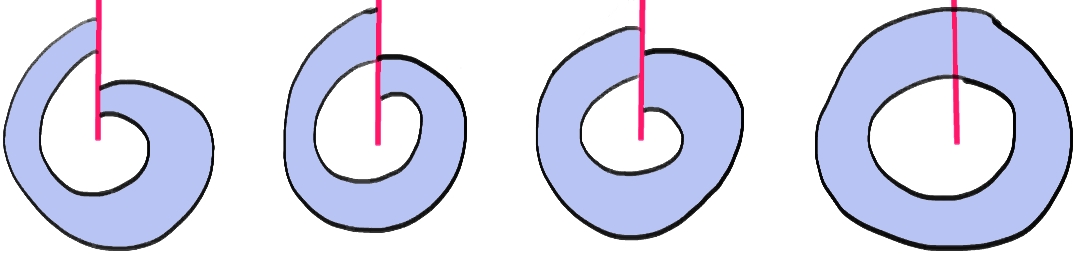}
\caption{Snail, spirals, and a cylinder.}
\label{fig:abands}
\end{figure}

\begin{definition}\label{def:jmplgt}
The \emph{jump length} of a band $B$ is equal to the jump length of any of its constituent horizontal paths.
\end{definition}

\begin{remark}\label{rmk:annularjump}
The jump length of an annular band is at most $12|\chi|$, since annular bands project to semi-embedded trajectories in $\Gamma$, whose length is $\leq 12|\chi|$ as noted in Remark~\ref{rem:semi}. 
\end{remark}

\begin{lemma}\label{lem:annular}
There exist $\mathcal{M}_a(\delta,n,g), \mathcal{M}_b(\delta,n,g) >0$ such that for every annular band $A$, each component of $\partial A$ lifts to a $(\mathcal{M}_a, \mathcal{M}_b)$-quasigeodesic.
\end{lemma}

\begin{proof} We handle the cases of snails and spirals simultaneously.

The subgeodesic $\iota_{_A}\cup \tau_{_A}$ is a concatenation $uvu'$ where $v=\iota_{_A}\cap \tau_{_A}$ and $u,u'$ are the components of $\iota_{_A}\cup \tau_{_A} - v$. Let $A$ be bounded by the two horizontal paths $h$ and $h'$, so $\partial A=\vartheta \sqcup \vartheta'$ where $\vartheta=uh$ and $\vartheta'=u'h'$. 
Let $v$ be the subarc of $\gamma$  connecting $\iota_{_A}$ to $\tau_{_A}$, and let $A$ be bounded by the two horizontal paths $h$ and $h'$ as in the spiral case. Then $\partial A=\vartheta \sqcup \vartheta' $ where  $\vartheta=\iota_{_A}v h$ and $\vartheta'=\tau_{_A} v h'$. 

We claim that, in both the spiral and snail cases, $\tilde \vartheta$ and $\tilde \vartheta'$ are uniform quasigeodesics in $\widetilde X$.
We prove this for $\tilde \vartheta$, and the proof for $\tilde \vartheta'$ is analogous.

The proof is an application of  Theorem~\ref{thm:loco}, and is identical to that of Lemma~\ref{lem:quasig} with the lifts of $h$ replacing $\{\sigma_i\}$, the lifts of $u$  replacing $\{\lambda_i\}$ in the spiral case, the lifts of $\iota_{_A} v$ replacing $\{\lambda_i\}$ in the snail case, and $L''>K'12|\chi|$. The proof uses that subrays of the geodesics containing consecutive lifts of $\lambda_i$ represent distinct points on $\partial \widetilde X$. This holds by Remark~\ref{rmk:separate}. 

Hence, $\tilde \vartheta$ and $\tilde \vartheta'$ are $(\mathcal{M}_a, \mathcal{M}_b)$-quasigeodesics.
\end{proof}

As a consequence of Lemma~\ref{lem:annular} and Proposition~\ref{prop:BHconj}, we obtain:

\begin{corollary}\label{cor:spiral}
There is a constant $\mathcal{M}_c$ such that for every spiral $A$, there is an annular diagram $A' \rightarrow X$  with $\area(A')\leq \mathcal{M}_c |\partial A'|$ and having $\partial A\cong \partial A'$ such that the following diagram commutes:
\[\begin{tikzcd}
\partial A' \arrow[rr, "\cong"] \arrow[rrd] &  & \partial A \arrow[d] \\
                                            &  & X                               
\end{tikzcd}\]
\end{corollary}

\begin{lemma}\label{cor:thickness}
There exists $\kappa(\delta,n,g) >0$, and a genus~$g$ diagram $\Sigma'$ with $\partial \Sigma'=\partial \Sigma$
 and such that every annular band in $\Sigma'$  has thickness at most $\kappa$.
\end{lemma}

\begin{proof}
 We deal with snails first. Let $A$ be a snail and let $\widetilde A$ be a component of its preimage in $\widetilde X$. We retain the notation of Lemma~\ref{lem:annular}.  
 
 Since $\tilde \vartheta$ and $\tilde \vartheta'$ are uniform quasigeodesics that stay at constant distance from each other, $\tilde \vartheta$ lies in the $\kappa_0$-neighbourhood of $\tilde \vartheta'$ in $\widetilde X$ for some $\kappa_0(\delta, \mathcal{M}_a\mathcal{M}_b)>0$.
 Let $p \in \tilde \vartheta$ and $q\in \tilde\vartheta'$ be points at distance at most $\kappa_0$ from each other and let $\overline{pq}$ be a geodesic from $p$ to $q$. Let $r \in \tilde\vartheta'$ be such that $d(p,r)=$ thickness of $A$.
 The path $\overline{pqr}$ is a uniform quasigeodesic. Indeed, let $\tilde{v}$ be a lift of $v$, and choose $\tilde{v}$ so that it is closest to $\overline{pr}$ among all lifts. Let $p'\in \tilde{\vartheta},r'\in \tilde{\vartheta}'$ be the endpoints of $\tilde v$, then $\tilde{v}$ is parallel to $\overline{pr}$ and $d(p,p')=d(r,r')\leq 12|\chi|$. Hence $\overline{pqr}$ lies in the $12|\chi|$-neighbourhood of  $\tilde{\vartheta}$, which is a $(\mathcal{M}_a,\mathcal{M}_b)$-quasigeodesic by Lemma~\ref{lem:annular}. 

 Since $\overline{pqr}$ is a quasigeodesic that stays at a uniformly bounded distance $\kappa_0$ from the geodesic $\overline{pq}$, it follows that $d(p,r)=|\overline{pr}|$ is uniformly bounded by some constant $\kappa'(\delta,n,g)$. Therefore the thickness of $A$ is at most $\kappa'$.
 
 For spirals, the proof is identical when $\iota_{_A}\cap\tau_{_A}$ is a single point. In the general case the situation is a little more subtle, since part of the interval that we need to bound lies in the interior of $A$ rather than on its boundary, which prevents us from knowing a priori that the path $\overline{pqr}$ is a quasigeodesic. To remedy this, replace $\Sigma$ by the diagram $\Sigma'$ containing the  bounded area annular diagram $A'$ obtained in Corollary~\ref{cor:spiral}. 
 Argue as in the previous case to see that the thickness of $A$ is at most $\kappa''(\delta,n,g)$.
 
 The claim follows by letting $\kappa=max\{\kappa',\kappa''\}$.
\end{proof}

\begin{figure}%[h!]
\centerline{\includegraphics[scale=0.56]{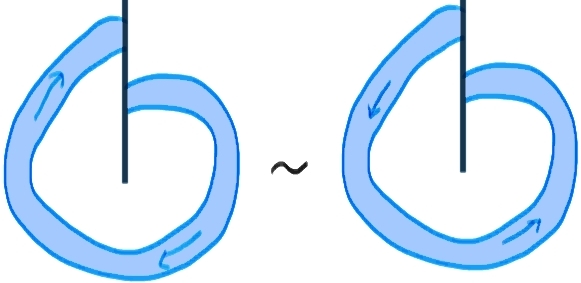} \hspace{1cm} \includegraphics[scale=0.59]{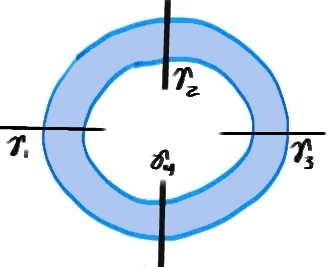}}
\caption{The two bands on the left are equivalent; the bands on the right, starting at any of $\gamma_1,\gamma_2,\gamma_3,\gamma_4$ are equivalent.}
\label{fig:equiv}
\end{figure}

\begin{definition}\label{def:bridges}
The \emph{interior} $\text{Int}(A)$ of a band is the interior of its image in $\Sigma$.
The \emph{generalised boundary} of $\Sigma$ is $\partial\Sigma-\cup_{A \in \mathcal{A}} \text{Int}(A)$. The generalised boundary of $\Sigma$ inherits the structure of $\partial \Sigma$, in the sense that each boundary curve is the concatenation of a vertical subgeodesic and a horizontal path.

A \emph{bridge} is a maximal band  in $\Sigma- \cup_{A \in \mathcal{A}} \text{Int}(A)$. It does not properly factor through another band in $\Sigma- \cup_{A \in \mathcal{A}} \text{Int}(A)$.
\end{definition}

Semi-embedded bands $B,B'$ are \emph{equivalent} if they have the same image in $\Sigma$.
Let $\mathcal{E}$ be the set of semi-embedded bands. Let $\mathcal{A}\subset \mathcal{E}$ be the equivalence classes having annular representatives.  
In view of the following Lemma, let $\mathcal{B}\subset \mathcal{E}$ denote the subset of equivalence classes represented by bridges.

\begin{lemma}\label{lem:bridgesemi}
 Let $B$ be a bridge, then $T_B \rightarrow \Gamma$ is semi-embedded.
\end{lemma}

\begin{proof}
We will show that if $T_B \rightarrow \Gamma$ is not semi-embedded, then $B$ contains an annular band, contradicting that $B \subset (\Sigma- \cup_{A \in \mathcal{A}} \text{Int}(A))$.

Let $T_B \rightarrow \Gamma$ be a trajectory that is not semi-embedded. There exists a subpath  $p=e w e \subset T_B$, where $w$ is  semi-embedded and does not traverse $e$ (but $w$ could traverse $e^{-1}$). The band $A$ with  $T_A=e w$ starts and ends on the same vertical $\gamma$, has semi-embedded trajectory, and has first and last jumps lying on distinct ideal triangles. Hence, $A$ is an annular band.
\end{proof}

\begin{remark}\label{rmk:jumplengthbridge}In view of Lemma~\ref{lem:bridgesemi} and Remark~\ref{rem:semi}, the jump length of a bridge is at most $12|\chi|$. 
\end{remark}

\begin{remark}\label{rmk:lengthallbands}
By Lemma~\ref{lem:flatland}, $|J| \leq K'$ for any jump $J$. By Remarks~\ref{rmk:annularjump} and~\ref{rmk:jumplengthbridge}, jump lengths of annular bands and bridges are at most $12|\chi|$. Therefore $|B|\leq 12K'|\chi|$ for every band $B \in \mathcal{A}\cup \mathcal{B}$.
\end{remark}

\begin{figure}[h!]
\centerline{\includegraphics[scale=0.56]{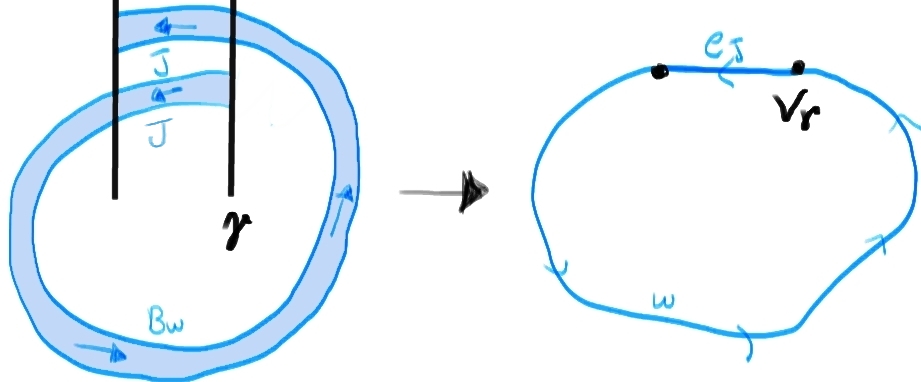}}
\caption{A band and its corresponding long trajectory.}
\label{fig:semiemb}
\end{figure}

\begin{remark}\label{rem:counts}
From Lemma~\ref{lem:numberofsemis} it follows that $\#\mathcal{E}\leq (12|\chi|+1)!$. Hence $\#\mathcal{A}+\#\mathcal{B}\leq (12|\chi|+1)!$ since  $\mathcal{A} \cap \mathcal{B}= \emptyset$.
\end{remark}

\begin{lemma}\label{lem:thickbridge} $\sum_{[B] \in \mathcal{B}} \text{thickness}(B) \ \leq \ \frac{1}{2}(|\partial \Sigma| + 2\kappa\#\mathcal{A})$.
\end{lemma}

\begin{proof}
Let $V\cup H=\partial(\Sigma- \cup_{A \in \mathcal{A}} Int(A))$, where $H$ consists of all horizontal subpaths of $\partial(\Sigma- \cup_{A \in \mathcal{A}} Int(A))$ and $V$ consists of the vertical subgeodesics of $\partial(\Sigma- \cup_{A \in \mathcal{A}} Int(A))$, and $V=\cup_j \iota_{B_j}\bigcup \cup_j  \tau_{B_j}$.
Hence $\sum_j |\iota_{B_j}| + \sum_j  |\tau_{B_j}|=|V|\leq |\partial \Sigma| + 2\kappa \#\mathcal{A}$.
\end{proof}

\begin{theorem}\label{cor:cutintodisc}
Suppose $\Sigma$ has no twisted spirals. There exists a genus~$g$ diagram $\Sigma^{\flat}$ with $\partial \Sigma^{\flat}= \partial \Sigma$ and 
 a constant $k_{gn}$ such that  $\area(\Sigma^{\flat})\leq k_{gn} |\boundary S| $.
\end{theorem}

Of course, there are no twisted spirals when $\Sigma$ is orientable. Theorem~\ref{cor:cutintodisc} in this generality to facilitate the proof of the main theorem in Section~\ref{sec:nono}.

\begin{proof}
Consider the annular bands, intriangles, and bridges that constitute $\Sigma$:

\begin{enumerate}

\item Let $\mathcal{A}=\{A_1, \ldots, A_{k-1}, A_k, \ldots, A_{\#\mathcal{A}}\}$, where $A_1, \ldots, A_{k-1}$ are snails and $A_k, \ldots, A_{\#\mathcal{A}}$ are spirals.
Then by Remark~\ref{rmk:lengthallbands} and Lemma~\ref{cor:thickness}, $\area(A_i)\leq 12|\chi|\kappa K'$ for every snail $A_i$.
By Corollary~\ref{cor:spiral} there are annular diagrams $A'_k, \ldots, A'_{\#\mathcal{A}}$ satisfying $\area(A'_i)\leq \mathcal{M}_c|\partial A'_i|$ and $\partial A_i=\partial A'_i$ for $i \in \{k, \ldots, \#\mathcal{A}\}$. Hence $\sum^{k-1}_{i=1} \area(A_i)+ \sum^{\#\mathcal{A}}_{i=k} \area(A'_i)\leq \kappa\#\mathcal{A}(12|\chi|K'+\mathcal{M}_c\sum_i|\partial A'_i|)$.

\item The side length of each intriangle $\Lambda$ is $\leq K'$ by Lemma~\ref{lem:flatland}. Thus Theorem~\ref{thm:discs} provides a constant $\mathcal{M}_\triangle$ and a disc diagram $\Lambda'$ with  $\partial \Lambda=\partial \Lambda'$. As there are $2|\chi|$ intriangles, $\sum_{\Lambda} \area( \Lambda')\leq K'\mathcal{M}_\triangle 2|\chi|$.

\item  Each bridge has length at most $12|\chi|K'$ by Remark~\ref{rmk:jumplengthbridge}.  
By Remark~\ref{rem:counts} and Lemma~\ref{lem:bridgesemi}, the sum of the thicknesses of the bridges is bounded above by $\frac{1}{2}(|\partial \Sigma| + 2\kappa\#\mathcal{A})$ by Lemma~\ref{lem:thickbridge}. 
Hence $\sum_{B \in \mathcal{B}}\area(B) \leq \frac{12}{2}K'|\chi|  \#\mathcal{B}(|\partial \Sigma| + 2\kappa\#\mathcal{A}) $.

\end{enumerate}
$\Sigma^{\flat}$  is obtained from $\Sigma$ by replacing each spiral $A$ with $A'$ and replacing each intriangle $\Lambda$ by $\Lambda'$.
 Then  $\partial \Sigma^{\flat}=\partial \Sigma$, which suffices by Remark~\ref{rmk:restrict}.  
 Finally:  
 \[\area(\Sigma^{\flat})\leq  k_{gn} |\boundary S| \] where
 \[k_{gn} \ = \ \#\mathcal{A}(12\kappa K'|\chi|+\mathcal{M}_c|\partial A'_i|) +  2K'\mathcal{M}_\triangle|\chi| +6K'\#\mathcal{B}|\chi| (1 + 2\kappa\#\mathcal{A}). \qedhere \]
\end{proof}

\section{A simple proof in a special case}\label{sec:special}

This section
proves a linear isoperimetric function
for 2-complexes that satisfy the strict weight test
 \cite{PrideHyperbolicComplexes88,GerstenReducible87}.
This was first explained for disc diagrams by Gersten.
We recall the  Combinatorial Gauss Bonnet Theorem and its associated formulas, and refer to Gersten and Pride as above for proofs,
 or to \cite{McCammondWiseFanLadder}  for the slight generalisation we use.
 
The curvatures of vertices and 2-cells are defined as follows:

$\curvature(x)=2\pi-\sum_{c\in  \text{Corners}(x)} \sphericalangle(c)-\pi\chi(\link(x))$

$\curvature(f)=2\pi-\sum_{c\in \text{Corners}(f)} \deficiency(c)$, where $\deficiency(c)=\pi -\sphericalangle(c)$
\begin{theorem} 
[Combinatorial Gauss-Bonnet] 
\label{thm:CGB}
Let $Y$ be a compact 2-complex with an angle assigned at each corner of each 2-cell, then $$2\pi\euler(Y) = \sum_{v \in  \text{Vertices}(Y)} \curvature(v)+ \sum_{f \in \text{$2$-cells}(Y)} \curvature(f)$$

\end{theorem}
We have in mind the case where $Y$ is a (possibly singular) surface.

\begin{defn}\label{def:coarse}
An \emph{angled 2-complex} $X$ is a 2-complex with an angle $\sphericalangle(c) \in \reals$ assigned to each corner $c$ of each 2-cell. (Equivalently, an angle is assigned to each edge in the link of each 0-cell.)

A map $Y\rightarrow X$ between 2-complexes is a \emph{near-immersion} if it is  a local-injection outside $Y^0$.
The angles of $X$ are pulled back to $Y$,
so  $Y$ is itself an angled 2-complex.
\end{defn}

\begin{definition}[Strict weight test]Let $\sigma \rightarrow \Lambda$ be a combinatorial path in a graph having an angle $\sphericalangle(c)$ for each edge $c$. Define $|\sigma|_\sphericalangle = \sum_{i=1}^n \sphericalangle(c_i)$ where $\sigma=c_1 \cdots c_n $ is a concatenation of edges.

The angled 2-complex $X$ satisfies the \emph{strict weight test} if the following hold:
\begin{enumerate}
\item $\kappa(f) < 0$ for each 2-cell $f$ and 
\item $|\sigma|_\sphericalangle> 2\pi$ for each essential closed combinatorial path $\sigma \rightarrow \link(x)$
\end{enumerate}
\end{definition}

\begin{lemma}\label{lem:exists}
Let $X$ be a compact angled 2-complex satisfying the strict weight test.
 There are constants $\nu, \xi$ such that the following holds.  Let $S \rightarrow X$ be a near-immersion of a surface diagram, and pullback the angles of $X$ to $S$. Then
 \begin{enumerate}
\item $\kappa(x) < \nu$ for any vertex $x$ in the interior of $S$
\item $\kappa(x)<\xi$ for any vertex in $\partial S$. 
\end{enumerate} 
\end{lemma}

\begin{proof}
For a non-singular vertex $x$ in $S$, the $\kappa(x)$ is computed directly from the path or cycle $\sigma$ corresponding to $\link(x)$. 
Namely, $\kappa(x)=2\pi-|\sigma|_\sphericalangle$ or $\kappa(x)=\pi-|\sigma|_\sphericalangle$ depending on whether $x$ is internal or not.
Note that when we pullback the angles at corners of $X$ to the corners of $S$, this assigns angles to edges of $\link(x)$ which are then associated to the angles of edges of $\sigma$. While $|\sigma|_\sphericalangle$ is more than $2\pi$ or more than $0$ when all angles are strictly positive, a little more effort is required when allowing arbitrary angles.

Recall that a path is \emph{semi-embedded} if it traverses each edge at most once in each direction,
and that there are finitely many semi-embedded paths in a finite graph. 

Any cycle $\sigma$ can be decomposed into semi-embedded cycles $\sigma_1 \cdots \sigma_n$. Since $|\sigma|_\sphericalangle=\sum |\sigma_i|_\sphericalangle$, we have that $\kappa(x)< 2\pi-\nu_0$ when $x$ is internal, and $\nu_0$ is the minimum $|\mu|_\sphericalangle$ where $\mu$ is a semi-embedded closed cycle.

Any immersed path can be decomposed as either $\lambda_1$ or $\lambda_1\sigma_1 \cdots \sigma_n \lambda_2$ where each $\lambda_i, \sigma_j$ are semi-embedded and each $\sigma_i$ is an immersed cycle. Therefore $\kappa(x)=\pi-|\sigma|_\sphericalangle < \pi- \xi_0$ when $x$ is a vertex on $\partial S$, and $\xi_0$ is twice the minimum $|\lambda_i|_\sphericalangle$.

The singular case is similar.
\end{proof}

For a genus~$g$ diagram $S$ with boundary circles $\{C_i\}$
we let $|\boundary S| = \sum|C_i|$.

\begin{proposition}\label{lem:b-c-npc} Let $X$ be a compact angled 2-complex with negative curvature.
There exists $K\geq0$ with the following property:

Let $S$ be a
 surface 
 diagram. 
Then
$\area(S)\leq K|\boundary S|$
for any near-immersion.
\end{proposition}

\begin{proof}

Let $g$ be the genus of $S$ and $n$ be the number of boundary components.

By Theorem~\ref{thm:CGB} we have:
$$2\pi(2-2g-n)=2\pi \euler(S) \ = \ \sum \curvature(v) + \sum \curvature(f)$$
so
$$ 2\pi(2-2g-n) \ \leq  \ \nu |\boundary S|-\xi \area(S)$$
where $\nu$ is an upperbound on positive curvature at a boundary vertex
and $-\xi <0$ is an upperbound on the negative curvature of a 2-cell (i.e., $|\xi|$ is the maximum in absolute value). Since $X$ is compact, such $\nu$ and $\xi$ always exist by Lemma~\ref{lem:exists}.  
Hence
$$\area(S)
\ \leq \
\frac{\nu}{\xi} |\boundary S|+ 2\pi(2g+n-2) 
\ \leq \
 |\boundary S|(\frac{\nu}{\xi}+4\pi g +2\pi)$$
as $n \leq |\boundary S|$, since each boundary circle of $S$ has at least one edge.
\end{proof}

\section{Nonorientable surface diagrams}\label{sec:nono}

In this section we explain how to generalise the proof of Theorem~\ref{thm:main} to non-orientable surface diagrams.
 As in the orientable case, the surface diagram decomposes into a union of spirals, snails, and bridges, but because of nonorientability some spirals may be twisted. We need to re-quantify the spiral bounds to generalise the proof.

\begin{proof}[(Proof of Theorem~\ref{thm:main} in the nonorientable case)]
Cases~\eqref{item:1} through~\eqref{item:5} of the proof of Theorem~\ref{thm:main} hold without any modification. The degenerate case of a Mobius strip is handled below in Proposition~\ref{prop:mobiuslinear}. To handle the non-degenerate cases, we assume, as in the orientable case, that every boundary circle of $S$ maps to a conjugacy class of an infinite-order element of $\pi_1 X$, and that $S$ is homeomorphic to a surface, so that its interior has an ideal triangulation. Now we replace our original diagram by a new infinite nonorientable surface diagram $S^\delta$, which in turn we replace with a buffered and trimmed nonorientable surface diagram $\Sigma$ (see Construction~\ref{cnst:buffer} and Construction~\ref{ct:trim}).

 While the constants change, all of the technical results in the previous sections hold for $\Sigma$. 
The only result that seems to utilise orientability is Lemma~\ref{lem:annular} -- but it actually uses the no-twisted-cylinders. We now explain how to navigate around it.  

For each twisted spiral $B_i$ in $\Sigma$, let $\mathbf{M}_i \subset B_i$ be the twisted cylinder of Lemma~\ref{lem:twisted}. Let $\Sigma_o=\Sigma - \cup_i \text{Int}(\mathbf{M}_i)$. Since twisted spirals are semi-embedded, $|\partial \mathbf M_i|\leq 12|\chi|K'$ by Remark~\ref{rmk:lengthallbands}. Thus:
 
$$ |\partial \Sigma_o| \ = \ |\partial \Sigma|+\sum_i|h_i| \ \leq \ |\partial \Sigma| +12K'|\chi|\#\mathcal{A}.$$

Since $\Sigma_o$ has no twisted spirals, all annular bands in $\Sigma_o$ are orientable. Theorem~\ref{cor:cutintodisc} provides a  diagram $\Sigma'_o$ with $\partial \Sigma'_o= \partial \Sigma_o$ and $\area (\Sigma'_o) \leq k_{g'n'} |\partial \Sigma_o|$ where $k_{g'n'}$ is a constant depending on the genus and number of boundary circles of $\Sigma'_o$. 

Applying Corollary~\ref{prop:mobiuslinear} to each $\mathbf{M}_i$ yields a twisted cylinder $\mathbf{M}'_i$ with $\partial \mathbf{M}_i=\partial \mathbf{M}'_i$ and $\area(\mathbf{M}'_i)\leq \mathbf{K}|\partial \mathbf{M}_i|$. A surface diagram $\Sigma'$ with the same homotopy type and with $\partial \Sigma'=\partial \Sigma$ is obtained by gluing the twisted cylinders along their boundaries:
\[\Sigma' \ = \ \Sigma'_o \bigcup_{\cup_i \partial\mathbf{M}'_i} \cup_i \mathbf{M}'_i \]
Finally 
\[ \area(\Sigma') \ = \ \area(\Sigma'_o) + \sum_i \area(\mathbf{M}'_i) \ \leq \ k_{g'n'} |\partial \Sigma_o| + \mathbf{K}|\partial \mathbf{M}|\#\mathcal{A}. \qedhere \]
\end{proof}

The linear isoperimetric inequality for Möbius diagrams follows readily from Proposition~\ref{prop:BHconj}:

\begin{corollary}\label{prop:mobiuslinear}
Let $X$ be a 2-complex with $\pi_1X$ $\delta$-hyperbolic. There is a constant $\mathbf{K}=\mathbf{K}(\delta)$ such that if an essential loop $\sigma \rightarrow X$ bounds a Möbius diagram $\mathbf{M} \rightarrow X$, then there exists a Möbius diagram $\mathbf{M}' \rightarrow X$ with $\partial \mathbf{M}=\partial \mathbf{M}'$ and $\area(\mathbf{M}')\leq \mathbf{K}|\partial \mathbf{M}|$.
\end{corollary} 

\begin{proof}
Let $\gamma \rightarrow X$ be a minimal length closed combinatorial path homotopic to a generator of $\pi_1 \mathbf{M}$, so  $\gamma^2$ is homotopic to $\sigma$. Note that $|\gamma^2|\leq R|\sigma|$ for some uniform $R$. By Proposition~\ref{prop:BHconj}, there is an annular diagram $C \rightarrow X$ with $\partial C= \gamma \sqcup \sigma$ and $\area(C)\leq \mathbf{M}|\partial C|= M R |\sigma|$ for some uniform  $M$. Hence $\area(C)\leq 2MR|\sigma|$. Let $\mathbf{M}'$ be the quotient of $C$ obtained by identifying $\gamma^2$ with itself by the $\integers ^2$-action.  
\end{proof}

We require the following Lemma:

\begin{lemma}\label{lem:twisted}
Every twisted spiral in $\Sigma$ contains a twisted cylinder whose boundary is a closed horizontal path.
\end{lemma}

\begin{proof}
Let $B$ be a twisted spiral. If $\iota_{_B} = \tau_{_B}$ then  $B$ is the desired twisted cylinder. Suppose $\iota_{_B} \neq \tau_{_B}$.  Let $\iota_{_B} \cup \tau_{_B}= u v w$ as in Figure~\ref{fig:twist}. As the thickness of the band is constant, $|u|=|w|$.  Consequently, the band starting at $u$ must terminate at $w$, removing this band yields the desired twisted cylinder.
\end{proof}

\begin{figure}%[h!]
\centerline{\includegraphics[scale=.7]{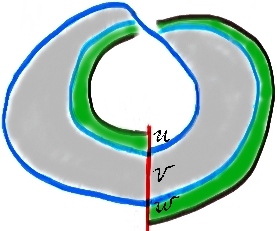}}
\caption{A twisted cylinder inside a twisted spiral.}
\label{fig:twist}
\end{figure}

{\bf Acknowledgement: } We are grateful to Piotr Przytycki for helpful comments.

%\begin{thebibliography}{}                                                                                                

%\end{thebibliography}
\bibliographystyle{alpha}
%\bibliographystyle{plain}
%\bibliography{../wise}
\bibliography{linearisopbob}

\def\cprime{$'$} \def\polhk#1{\setbox0=\hbox{#1}{\ooalign{\hidewidth
  \lower1.5ex\hbox{`}\hidewidth\crcr\unhbox0}}} \def\cprime{$'$}
  \def\cprime{$'$} \def\polhk#1{\setbox0=\hbox{#1}{\ooalign{\hidewidth
  \lower1.5ex\hbox{`}\hidewidth\crcr\unhbox0}}}
\begin{thebibliography}{ABDY13}

\bibitem[ABDY13]{ABDY13}
Aaron Abrams, Noel Brady, Pallavi Dani, and Robert Young.
\newblock Homological and homotopical {D}ehn functions are different.
\newblock {\em Proc. Natl. Acad. Sci. USA}, 110(48):19206--19212, 2013.

\bibitem[BH99]{BridsonHaefliger}
Martin~R. Bridson and Andr{\'e} Haefliger.
\newblock {\em Metric spaces of non-positive curvature}.
\newblock Springer-Verlag, Berlin, 1999.

\bibitem[BH05]{BridsonHowie05}
Martin~R. Bridson and James Howie.
\newblock Conjugacy of finite subsets in hyperbolic groups.
\newblock {\em Internat. J. Algebra Comput.}, 15(4):725--756, 2005.

\bibitem[BH13]{BH13}
D.~J. Buckley and Derek~F. Holt.
\newblock The conjugacy problem in hyperbolic groups for finite lists of group
  elements.
\newblock {\em Internat. J. Algebra Comput.}, 23(5):1127--1150, 2013.

\bibitem[Bri02]{Bridson2002}
Martin~R. Bridson.
\newblock The geometry of the word problem.
\newblock In {\em Invitations to geometry and topology}, volume~7 of {\em Oxf.
  Grad. Texts Math.}, pages 29--91. Oxford Univ. Press, Oxford, 2002.

\bibitem[BT02]{BurilloTaback2002}
Jos\'{e} Burillo and Jennifer Taback.
\newblock Equivalence of geometric and combinatorial {D}ehn functions.
\newblock {\em New York J. Math.}, 8:169--179, 2002.

\bibitem[Ger87]{GerstenReducible87}
S.~M. Gersten.
\newblock Reducible diagrams and equations over groups.
\newblock In {\em Essays in group theory}, pages 15--73. Springer, New
  York-Berlin, 1987.

\bibitem[Ger96]{Gersten1996}
S.~M. Gersten.
\newblock Subgroups of word hyperbolic groups in dimension {$2$}.
\newblock {\em J. London Math. Soc. (2)}, 54(2):261--283, 1996.

\bibitem[Ger98]{Gersten1998}
S.~M. Gersten.
\newblock Cohomological lower bounds for isoperimetric functions on groups.
\newblock {\em Topology}, 37(5):1031--1072, 1998.

\bibitem[Gro87]{Gromov87}
M.~Gromov.
\newblock Hyperbolic groups.
\newblock In {\em Essays in group theory}, volume~8 of {\em Math. Sci. Res.
  Inst. Publ.}, pages 75--263. Springer, New York, 1987.

\bibitem[HW15]{HagenWiseFreeByZgeneral}
Mark~F. Hagen and Daniel~T. Wise.
\newblock Cubulating hyperbolic free-by-cyclic groups: the general case.
\newblock {\em Geom. Funct. Anal.}, 25(1):134--179, 2015.

\bibitem[Lan00]{Lang2000}
Urs Lang.
\newblock Higher-dimensional linear isoperimetric inequalities in hyperbolic
  groups.
\newblock {\em Internat. Math. Res. Notices}, (13):709--717, 2000.

\bibitem[Min00]{Mineyev2000}
Igor Mineyev.
\newblock Higher dimensional isoperimetric functions in hyperbolic groups.
\newblock {\em Math. Z.}, 233(2):327--345, 2000.

\bibitem[MP16]{MPE16}
Eduardo Mart\'{\i}nez-Pedroza.
\newblock A note on fine graphs and homological isoperimetric inequalities.
\newblock {\em Canad. Math. Bull.}, 59(1):170--181, 2016.

\bibitem[MW02]{McCammondWiseFanLadder}
Jonathan~P. McCammond and Daniel~T. Wise.
\newblock Fans and ladders in small cancellation theory.
\newblock {\em Proc. London Math. Soc. (3)}, 84(3):599--644, 2002.

\bibitem[OS20]{OS20}
A.~Yu. Olshanskii and M.~V. Sapir.
\newblock Conjugacy problem in groups with quadratic {D}ehn function.
\newblock {\em Bull. Math. Sci.}, 10(1):1950023, 103, 2020.

\bibitem[Pri88]{PrideHyperbolicComplexes88}
Stephen~J. Pride.
\newblock Star-complexes, and the dependence problems for hyperbolic complexes.
\newblock {\em Glasgow Math. J.}, 30(2):155--170, 1988.

\bibitem[Thu82]{Thurston82}
William~P. Thurston.
\newblock Three-dimensional manifolds, {K}leinian groups and hyperbolic
  geometry.
\newblock {\em Bull. Amer. Math. Soc. (N.S.)}, 6(3):357--381, 1982.

\end{thebibliography}

\end{document}